\documentclass[10pt,final]{amsart}

\usepackage{newcent}

\usepackage[active]{srcltx}
\usepackage[utf8]{inputenc}
\usepackage[english]{babel}
\usepackage{graphicx}
\usepackage[T1]{fontenc}
\usepackage{amsmath}
\usepackage{amssymb} 
\usepackage{amsthm}
\usepackage{mathabx}
\usepackage{bbm} 
\usepackage{tikz}
\usepackage[numbers]{natbib}
\usepackage{array}
\usepackage[left=2.05cm,right=2.05cm,top=2.95cm,bottom=2.8cm]{geometry}
\usepackage[foot]{amsaddr} 

\usepackage{setspace}

\usepackage[pdfborder={0 0 0}]{hyperref}

\numberwithin{equation}{section} 

\newcommand{\R}{\mathbf{R}}
\newcommand{\N}{\mathbf{N}}
\newcommand{\F}{\mathcal{F}}

\newcommand{\Hen}{\mathcal{H}} 

\newcommand{\Prob}{\mathcal{P}}
\newcommand{\Proba}{\mathcal{P}_\mathrm{adm}} 
\newcommand{\Probin}{\mathcal{P}_\mathrm{in}} 
\newcommand{\Probbc}{\mathcal{P}_\mathrm{bc}} 
\newcommand{\Leb}{\mathcal{L}} 
\newcommand{\C}{\mathcal{C}} 
\newcommand{\A}{\mathcal{A}} 
\newcommand{\Afr}{\mathfrak{A}} 
\newcommand{\Pbc}{P_{\mathrm{bc}}} 

\newcommand{\vbf}{\mathbf{v}} 
\newcommand{\ddr}{~\mathrm{d}} 
\newcommand{\dr}{\partial}

\newcommand{\dst}[1]{\displaystyle{#1}}
\newcommand{\bsl}{\backslash}

\newtheorem{theo}{Theorem}[section]
\newtheorem*{theo*}{Theorem}
\newtheorem{prop}[theo]{Proposition}
\newtheorem{crl}[theo]{Corollary}
\newtheorem{lm}[theo]{Lemma}
\newtheorem{defi}[theo]{Definition}
\newtheorem{asmp}{Assumption}

\newcommand{\intervalle}[4]{\mathopen{#1}#2
                            \mathclose{}\mathpunct{},#3
                            \mathclose{#4}}
\newcommand{\intervalleff}[2]{\intervalle{[}{#1}{#2}{]}}
\newcommand{\intervalleof}[2]{\intervalle{(}{#1}{#2}{]}}
\newcommand{\intervallefo}[2]{\intervalle{[}{#1}{#2}{)}}
\newcommand{\intervalleoo}[2]{\intervalle{(}{#1}{#2}{)}}

\begin{document}

\title{Time-convexity of the entropy in the multiphasic formulation of the incompressible Euler equation}

\date{\today}
\author{Hugo Lavenant}
\address{Laboratoire de Math\'ematiques d'Orsay, Univ. Paris-Sud, CNRS, Universit\'e Paris-Saclay, 91405 Orsay Cedex, France}
\email{hugo.lavenant@math.u-psud.fr}


\maketitle

\begin{abstract}
We study the multiphasic formulation of the incompressible Euler equation introduced by Brenier: infinitely many phases evolve according to the compressible Euler equation and are coupled through a global incompressibility constraint. In a convex domain, we are able to prove that the entropy, when averaged over all phases, is a convex function of time, a result that was conjectured by Brenier. The novelty in our approach consists in introducing a time-discretization that allows us to import a \emph{flow interchange} inequality previously used by Matthes, McCann and Savar\'e to study first order in time PDE, namely the JKO scheme associated with non-linear parabolic equations.        
\end{abstract}

\section{Introduction and statement of the main result}

Since the idea of Arnold \cite{Arnold1966} to consider the motion of an incompressible and inviscid fluid, described by the Euler equation, as a variational problem, namely a geodesic on the (infinite dimensional) group of orientation and measure-preserving diffeomorphisms (this is formally speaking an instance of the \emph{least action principle}), this variational point of view has turned out to be fruitful. In particular, Brenier introduced relaxations leading to generalized geodesics on the group of measure-preserving maps: translated at a microscopic level, fluid particles are allowed to split and diffuse on the whole space (for a general survey, see for instance \cite{Daneri2012}). We will concentrate in this paper on one of Brenier's model with a flavor of Eulerian point of view introduced in \cite{Brenier1999} (see also \cite[Section 4]{Brenier2003}, \cite[Section 1.5.3]{Daneri2012} and \cite{Ambrosio2009}) which goes as follows.  

There are (possibly infinitely) many phases indexed by a parameter $\alpha$ which belongs to some probability space $(\Afr, \A, \theta)$. At a fixed time $t$, each phase is described by its density $\rho^\alpha_t$ and its velocity field $\vbf^\alpha_t$, which are functions of the position $x$. We assume that all the densities are confined in a fixed bounded domain $\Omega$, and up to a normalization constant $\rho^\alpha_t$ can be seen as a probability measure on $\Omega$. The evolution in time of the phase $\alpha$ is done according to the continuity equation 
\begin{equation}
\label{equation_continuity_equation}
\partial_t \rho_t^\alpha + \nabla \cdot ( \rho^\alpha_t \vbf^\alpha_t ) = 0,
\end{equation}  
where $\nabla \cdot$ stands for the divergence operator. We assume no-flux boundary conditions on $\partial \Omega$, thus the total mass of $\rho^\alpha$ is preserved over time. The different phases are coupled through the incompressibility constraint: at a fixed $t$ the density of all the different phases must sum up to the Lebesgue measure $\Leb$ (restricted to $\Omega$). In other words, for any $t$ we impose that 
\begin{equation}
\label{equation_incompressibility_condition}
\int_\Afr \rho^\alpha_t \ddr \theta(\alpha) = \Leb.
\end{equation} 
Looking at the problem from a variational point of view, we assume that the values of $\rho^\alpha_t$ are fixed for $t= 0$ and $t=1$ and that the trajectories observed are those solving the following variational problem: 
\begin{equation}
\label{equation_variationnal_problem_unformal}
\min \left\{ \int_\Afr \int_0^1 \int_\Omega \frac{1}{2} |\vbf^\alpha_t(x)|^2 \rho^\alpha_t(x) \ddr  x \ddr t \ddr \theta(\alpha) \ :  \ (\rho^\alpha, \vbf^\alpha) \text{ satisfies \eqref{equation_continuity_equation} and \eqref{equation_incompressibility_condition}} \right\}.
\end{equation}
From a physical point of view, the functional which is minimized corresponds to the average (over all phases) of the integral over time of the kinetic energy, namely the global \emph{action} of all the phases. Without the incompressibility constraint, each phase would evolve independently and follow a geodesic in the Wasserstein space joining $\rho^\alpha_0$ to $\rho^\alpha_1$. 

In Brenier's original formulation, the space $(\Afr, \mathcal{A}, \theta)$ is the domain $\Omega$ endowed with the Lebesgue measure $\Leb$. If $h : \Omega \to \Omega$ is a measure-preserving map, "classical" boundary conditions are those where $\rho^\alpha_0$ is the Dirac mass located at $\alpha$ and $\rho^\alpha_1$ is the Dirac mass located at $h(\alpha)$. In a classical solution, each phase $\alpha$ will be of the form $\rho^\alpha_t = \delta_{y^\alpha(t)}$, where $y^\alpha : \intervalleff{0}{1} \to \Omega$ is a curve joining $\alpha$ to $h(\alpha)$. But, even if one starts with "classical" boundary conditions, there are cases where the phase $\alpha$ may split and $\rho^\alpha$ may not be a Dirac mass for any $t \in \intervalleoo{0}{1}$, leading to a "non-classical" solution (for examples of such cases, the reader can consult \cite[Section 6]{Brenier1989} or the detailed study \cite{Bernot2009}).   

\bigskip

With formal considerations (see for instance \cite[Section 4]{Brenier2003}), one can be convinced that for each phase $\alpha$, the optimal velocity field is the gradient of a scalar field $\varphi^\alpha$ (i.e. $\vbf^\alpha_t = \nabla \varphi^\alpha_t$), and that each $\varphi^\alpha$ evolves according to a Hamilton-Jacobi equation 
\begin{equation*}
\partial_t \varphi^\alpha_t + \frac{|\nabla \varphi^\alpha_t|^2}{2} = -  p_t,
\end{equation*} 
with a pressure field $p$ that does not depend on $\alpha$ and that arises from the incompressibility constraint. If we look at the Boltzmann entropy of the phase $\alpha$, a lengthy formal computation leads to 
\begin{equation*}
\frac{\mathrm{d}^2}{\mathrm{d}t^2} \int_\Omega \rho^\alpha_t(x) \ln \rho^\alpha_t(x) \ddr x  = \int_\Omega \left[ \Delta p_t(x)  + |D^2 \varphi^\alpha_t(x)|^2  \right] \rho^\alpha_t(x) \ddr x,
\end{equation*}  
where $|D^2 \varphi^\alpha_t|^2 = \sum_{i,j} ( \partial_{ij} \varphi^\alpha_t)^2$ is the square of the Hilbert-Schmidt norm of the Hessian w.r.t. (with respect to) space of $\varphi^\alpha_t$. Thus, if one defines the averaged entropy $H$ as a function of time by 
\begin{equation*}
H(t) := \int_\Afr \left( \int_\Omega \rho^\alpha_t(x) \ln \rho^\alpha_t(x) \ddr x \right) \ddr \theta(\alpha),
\end{equation*}  
the previous computation leads to 
\begin{equation*}
H''(t)  = \int_\Omega \Delta p_t +  \int_\Afr \int_\Omega  |D^2 \varphi^\alpha_t(x)|^2 \rho^\alpha_t(x) \ddr x \ddr \theta(\alpha) \geqslant \int_{\partial \Omega} \nabla p_t \cdot n
\end{equation*}
where $n$ stands for the outward normal of $\Omega$. At this point, it becomes natural to assume that $\Omega$ is \emph{convex}. Indeed, if this is the case, the acceleration of a fluid particle located on the boundary will be directed toward the interior of $\Omega$ because the particle is constrained to stay in $\Omega$. As the acceleration of the fluid particles is -- at least heuristically -- equal to $-\nabla p$, it is reasonable to expect that $\nabla p \cdot n \geqslant 0$ on $\partial \Omega$. Therefore, at a formal level, assuming the convexity of $\Omega$ leads to $H'' \geqslant 0$, i.e. to the property that the averaged entropy $H$ is a convex function of time. This was remarked and conjectured by Brenier in \cite[section 4]{Brenier2003}, but has received no proof to our knowledge until now. Indeed, the main difficulty lies in the fact that \emph{a priori} the solutions are not regular enough to make the above computation rigorous. The goal of this paper is to give a rigorous statement and to prove this conjecture. The two main statements can informally be stated as follows: 

\begin{theo*}
Let us assume that $\Omega$ is convex. If $(\rho^\alpha, \vbf^\alpha)_{\alpha \in \Afr}$ is the unique solution of the variational problem \eqref{equation_variationnal_problem_unformal} whose total entropy $\int_0^1 H(t) \ddr t$ is finite and minimal compared to all other solutions, then $H$ is a convex function. 
\end{theo*}  

\begin{theo*}
Let us assume that $\Omega$ is convex and that the boundary terms are such that $H(0)$ and $H(1)$ are finite. Then there exists a solution of the variational problem \eqref{equation_variationnal_problem_unformal} such that $\int_0^1 H(t) \ddr t$ is finite and minimal compared to all other solutions.  
\end{theo*}

\noindent More precisely, see Theorem \ref{theorem_main_convexity_entropy} and Theorem \ref{theorem_finite_entropy_IF_implies_L1} for the exact assumptions and statements, and Section \ref{section_equivalence_formulation} for the translation in Brenier's parametric setting (see below).

Let us remark that the convexity of the entropy is invisible for classical solutions. Indeed, if $\rho^\alpha_t$ is a Dirac mass for any $\alpha$, then $H(t) = + \infty$. Thus the convexity of the entropy is non trivial only for "strongly" non-classical solutions.  

\bigskip

The strategy to prove the convexity of the entropy goes as follows. For a fixed $\alpha$, we see $t \mapsto \rho^\alpha_t$ as a curve in the space of probability measures on $\Omega$. This space (denoted $\Prob(\Omega)$) can be endowed with the $2$-Wasserstein distance $W_2(\cdot, \cdot)$ coming from optimal transport: the squared distance $W_2^2(\mu, \nu)$ between two measures $\mu$ and $\nu$ is just the optimal transport cost among all transport between $\mu$ and $\nu$. The interest of this distance is that the action $\int_0^1 \int_\Omega \frac{1}{2} |\vbf^\alpha_t(x)|^2 \rho^\alpha_t(x) \ddr x  \ddr t$ of the curve $\rho^\alpha$ (or at least, the minimal value of the action among all velocity field $\vbf^\alpha$ satisfying the continuity equation \eqref{equation_continuity_equation}) can be seen as the integral w.r.t. time of the square of the \emph{metric derivative} of the curve $t \mapsto \rho^\alpha_t$ in the metric space $(\Prob(\Omega), W_2)$. In particular, there appears a natural time-discretization of the action: if $N$ is large enough and $\tau := 1/N$ we expect that 
\begin{equation*}
\inf_{\vbf^\alpha \text{ satisfying } \eqref{equation_continuity_equation}} \  \int_0^1 \int_\Omega \frac{1}{2} |\vbf^\alpha_t(x)|^2 \rho^\alpha_t(x) \ddr x  \ddr t = \int_0^1 \frac{1}{2} \left| \dot{\rho}^\alpha_t \right|^2 \ddr t  \simeq \sum_{k=1}^N \frac{\tau}{2} \left( \frac{W_2( \rho^\alpha_{(k-1) \tau}, \rho^\alpha_{k \tau} )}{\tau} \right)^2.   
\end{equation*} 
At a discrete level, $\rho^\alpha$, which is a curve valued in $\Prob(\Omega)$, is approximated by an element of $\Prob(\Omega)^{N+1}$. The incompressibility constraint will be relaxed in order to allow comparison of the optimizer with any other competitor. If $\C_q : \Prob(\Omega) \to \R$ penalizes (more and more as $q \to + \infty$) the probability measures that are different from the Lebesgue measure, a discretized version of the Euler variational formulation \eqref{equation_variationnal_problem_unformal} might read
\begin{equation*}
\min \left\{ \int_\Afr \left[ \sum_{k=1}^N \frac{\tau}{2} \left( \frac{W_2( \rho^\alpha_{(k-1) \tau}, \rho^\alpha_{k \tau} )}{\tau} \right)^2 \right] \ddr \theta(\alpha) + \sum_{k=0}^N \C_q \left( \int_\Afr \rho^\alpha_{k \tau} \ddr \theta(\alpha) \right) \ : \ \rho^\alpha_0, \rho^\alpha_1 \text{ fixed for } \theta \text{-a.e. } \alpha \right\}.
\end{equation*}
Even though it would be possible, we will not write down the optimality conditions of this discretized problem as they contain much more information than needed for our goal. Instead, if we take the minimizer of the discretized problem, we will let the $k$-th component follow the flow of the heat equation (with no-flux boundary conditions) and use the result as a competitor. The key point is that the heat equation is strongly related to the Wasserstein distance: the heat flow is the gradient flow of the functional entropy $\rho \mapsto \int_\Omega \rho \ln \rho$ in the metric space $(\Prob(\Omega), W_2)$. In particular, and this is called the \emph{Evolution Variational Inequality}, one can estimate the derivative of the Wasserstein distance along the heat flow $\Phi_s$: 
\begin{equation*}
\left. \frac{\ddr}{\ddr s} \frac{W_2^2(\Phi_s \mu, \nu)}{2} \right|_{s = 0} \leqslant \int_\Omega \nu \ln \nu - \int_\Omega \mu \ln \mu. 
\end{equation*}       
This kind of inequality was previously used by Matthes, Mccann and Savaré in \cite{Matthes2009} under the name of \emph{flow interchange} to tackle first order (in time) PDEs (basically discretizations via the JKO scheme of gradient flows). As far as we know, this kind of technique has never been used for second order PDEs like the Euler equation\footnote{Let us precise that this idea is adapted from an ongoing work \cite{Lavenant2017} with Filippo Santambrogio where the same kind of technique is used to provide regularity for solutions of quadratic Mean Field Games.}. With this evolution variational inequality, one can show in a very simple way that the averaged entropy $\int_\Afr \int_\Omega \rho^\alpha_{k \tau} \ln \rho^\alpha_{k \tau} \ddr \theta(\alpha)$ is a (discrete) convex function of $k$. Then, one can expect that the solutions of the discretized problem will converge to those of the original one, and that the convexity of the entropy will be preserved at the limit.

\bigskip

It happens that all the quantities involved do not really depend on the particular dependence of the $\rho^\alpha$ in $\alpha$. Indeed, if one denotes by $\Gamma$ the space of continuous curves valued in the probability measures on $\Omega$ endowed with the Wasserstein distance (in short $\Gamma = C(\intervalleff{0}{1}, \Prob(\Omega))$), everything only depends on the image measure of $\theta$ through the map $\alpha \mapsto \rho^\alpha$. The natural object we are dealing with is therefore a probability measure on $\Gamma$, something that one can call (by analogy with \cite{Bernot2005}) a $W_2$-traffic plan. In a way, the application $\alpha \mapsto \rho^\alpha$ is a parametrization of a $W_2$-traffic plan: that's why we will call Brenier's formulation the parametric one, while we will work in the non parametric setting, dealing directly with probability measures on $\Gamma$. In our setting, most topological properties are easier to handle, and notations are according to us simplified. Even though any probability measure on $\Gamma$ cannot be \emph{a priori} parametrized, we will show that it is the case for the solutions of Euler's variational problem. Therefore, our results can be translated in Brenier's parametric setting. 

\bigskip 

This paper is organized as follows. In Section \ref{section_preliminaries}, we introduce the notations that we will use, we briefly recall some properties of the Wasserstein distance and of the heat equation seen as the gradient flow of the entropy. We state explicitly the variational problem we are interested in and prove the existence of a solution (a result which is known since \cite{Brenier1989}). We give a rigorous statement of the theorems that we prove in the next two sections. Section \ref{section_discrete problem} introduces the discrete problem and proves the convexity of the averaged entropy at the discrete level. This section contains the key ingredient around which all the proof revolves. In Section \ref{section_discrete_to_continuous}, we show that the solutions of the discrete problems converge to the solution of the original one, and that the convexity of the averaged entropy is preserved when the limit is taken. Though lengthy and technical, this section does not contain profound ideas. Finally, Section \ref{section_equivalence_formulation} is devoted to the proof of the equivalence between the parametric and non parametric formulations. 

\section{Notations, preliminary results and existence of a solution to the continuous problem}
\label{section_preliminaries}

If $X$ is a polish space (complete, metric, separable), the space of Borel probability measures on $X$ will be denoted by $\Prob(X)$, and $C(X)$ is the space of continuous and bounded functions on $X$ valued in $\R$. The space $\Prob(X)$ will be endowed with the topology of the weak convergence of measures (i.e. the topology induced by the duality with $C(X)$). 

\bigskip

In all the sequel, we will denote by $\Omega$ a closed bounded \emph{convex} subset of $\R^d$ with non empty interior. In particular, $\Omega$ is compact. In order to avoid normalization constants, we assume that the Lebesgue measure of $\Omega$ is $1$. The Lebesgue measure on $\Omega$, which is therefore a probability measure, will be denoted by $\Leb$. 

\subsection{The Wasserstein space}

The space $\Prob(\Omega)$ of probability measures on $\Omega$ is endowed with the Wasserstein distance: if $\mu$ and $\nu$ are two elements of $\Prob(\Omega)$, the $2$-Wasserstein distance $W_2(\mu, \nu)$ between $\mu$ and $\nu$ is defined by 
\begin{equation}
\label{equation_definition_Wasserstein_distance}
W_2(\mu, \nu) := \sqrt{ \min \left\{ \int_{\Omega \times \Omega} |x-y|^2 \ddr \gamma(x,y) \ :  \ \gamma \in \Prob(\Omega \times \Omega) \text{ and } \pi_0 \# \gamma = \mu, \ \pi_1 \# \gamma = \nu \right\}  }.
\end{equation} 
In the formula above, $\pi_0$ and $\pi_1 : \Omega \times \Omega \to \Omega$ stand for the projections on respectively the first and second component of $\Omega \times \Omega$. If $T : X \to Y$ is a measurable application and $\mu$ is a measure on $X$, then the image measure of $\mu$ by $T$, denoted by $T \# \mu$, is the measure defined on $Y$ by $(T \# \mu)(B) = \mu(T^{-1}(B))$ for any measurable set $B \subset Y$. It can also be defined by 
\begin{equation*}
\int_Y a(y) \ddr (T \# \mu)(y) := \int_X a(T(x)) \ddr \mu(x), 
\end{equation*} 
this identity being valid as soon as $a : Y \to \R$ is an integrable function. 

For general results about optimal transport, the reader might refer to \cite{Villani2003} or \cite{SantambrogioOTAM}. We recall that $W_2$ defines a metric on $\Prob(\Omega)$ that metrizes the weak convergence of measures. Therefore, thanks to Prokhorov's Theorem, the space $(\Prob(\Omega), W_2)$ is a compact metric space. We also recall that $(\Prob(\Omega), W_2)$ is a geodesic space and if $\gamma \in \Prob(\Omega \times \Omega)$ is optimal in formula \eqref{equation_definition_Wasserstein_distance}, then a constant-speed geodesic $\rho : \intervalleff{0}{1} \to \Prob(\Omega)$ joining $\mu$ to $\nu$ is given by $\rho(t) := \pi_t \# \gamma$ with $\pi_t : (x,y) \in \Omega \times \Omega \to (1-t)x + ty \in \Omega$ (remark that the convexity of $\Omega$ is important here). Reciprocally, every constant-speed geodesic is of this form (see \cite[prop. 5.32]{SantambrogioOTAM}). We also recall that $W_2^2 : \Prob(\Omega) \times \Prob(\Omega) \to \R$ is a convex function.    

We will consider the entropy (w.r.t. the Lebesgue measure) functional $\Hen$ on $\Prob(\Omega)$. It is the functional $\Hen : \Prob(\Omega) \to \intervalleff{0}{+ \infty}$ defined by, for any $\mu \in \Prob(\Omega)$, 
\begin{equation}
\label{definition_entropy}
\Hen(\mu) := \begin{cases}
\dst{\int_\Omega \mu(x) \ln(\mu(x)) \ddr x} & \text{if } \mu \text{ is absolutely continuous w.r.t. } \Leb, \\
+ \infty & \text{else}.
\end{cases} 
\end{equation} 
The fact that $\Hen \geqslant 0$ on $\Prob(\Omega)$ is a consequence of Jensen's inequality and of the normalization $\Leb(\Omega) = 1$. We will also deal with an other functional on $\Prob(\Omega)$ that we will use to penalize the concentrated measures, namely the $q$-th power of the density. More precisely, if $q > 1$, we denote by $\C_q : \Prob(\Omega) \to \intervalleff{0}{+ \infty}$ the congestion functional defined by, for any $\mu \in \Prob(\Omega)$, 
\begin{equation}
\label{equation_definition_power}
\C_q(\mu) := \begin{cases}
\dst{\int_\Omega \mu(x)^q \ddr x - 1 }  & \text{if } \mu \text{ is absolutely continuous w.r.t. } \Leb, \\
+ \infty & \text{else}.
\end{cases} 
\end{equation} 
Again, thanks to Jensen's inequality, we see that $\C_q(\mu) \geqslant 0$ with equality if and only if $\mu = \Leb$. We recall that a functional is geodesically convex on $(\Prob(\Omega), W_2)$ if for any two given probability measures, there exists a constant-speed geodesic connecting these two measures along which the functional is convex. A functional will be called convex if it is so w.r.t. the usual affine structure on $\Prob(\Omega)$. Well known facts about $\Hen$ and $\C_q$ are summarized in the following proposition (see \cite[chap. 7]{SantambrogioOTAM}).    

\begin{prop}
\label{proposition_proprties_H_Cp}
For any $q > 1$, the functionals $\Hen$ and $\C_q$ are l.s.c. (lower semi-continuous), strictly convex and geodesically convex on $(\Prob(\Omega), W_2)$. 
\end{prop}

\noindent Let us underline that the convexity of $\Omega$ is needed to get the geodesic convexity of $\Hen$ and $\C_q$.

\subsection{Absolutely continuous curves in the Wasserstein space}

If $S$ is a closed subset of $\intervalleff{0}{1}$, $\Gamma_S$ will denote the set of continuous functions on $S$ valued in $\Prob(\Omega)$ (in practice, we will only consider subsets $S$ that have a finite number of points or that are subintervals of $\intervalleff{0}{1}$). In the case where the index $S$ is omitted, it is assumed that $S = \intervalleff{0}{1}$. This space will be equipped with the distance $d$ of the uniform convergence, i.e.
\begin{equation*}
d(\rho^1, \rho^2) := \max_{t \in S} W_2( \rho^1(t), \rho^2(t) ).
\end{equation*}
For any closed subset $S'$ of $S$, the application $e_{S'} : \Gamma_S \to \Gamma_{S'}$ is the restriction operator. In the case where $S'= \{ t \}$ is a singleton, we will use the notation $e_t := e_{\{ t \}}$ and often use the compact writing $\rho_t$ for $e_t(\rho) = \rho(t)$. One can see that $\Gamma_S$ is a polish space, and that it is compact if $S$ contains a finite number of points. 

Following \cite[Definition 1.1.1]{Ambrosio2005}, we give ourselves the following definition. 

\begin{defi}
We say that a curve $\rho \in \Gamma$ is $2$-absolutely continuous if there exists a function $f \in L^2(\intervalleff{0}{1})$ such that, for every $0 \leqslant t \leqslant s \leqslant 1 $,  
\begin{equation*}
W_2(\rho_t, \rho_s) \leqslant \int_t^s f(r) \ddr r.
\end{equation*}
\end{defi}

\noindent The main interest of this notion lies in the two following theorems that we recall. 

\begin{theo}
If $\rho \in \Gamma$ is a $2$-absolutely continuous curve, then the quantity 
\begin{equation*}
|\dot{\rho}_t| := \lim_{h \to 0} \frac{W_2(\rho_{t+h}, \rho_t)}{h}
\end{equation*}
exists and is finite for a.e. $t$. Moreover, 
\begin{equation}
\label{equation_representation_A_sup}
\int_0^1 |\dot{\rho}_t|^2 \ddr t = \sup_{N \geqslant 2} \ \ \sup_{0 \leqslant t_1 < t_2 < \ldots < t_N \leqslant 1} \ \ \sum_{k=2}^{N} \frac{W_2^2(\rho_{t_{k-1}}, \rho_{t_k})}{t_k - t_{k-1}}.
\end{equation}
\end{theo}

\begin{proof}
The first part is just \cite[Theorem 1.1.2]{Ambrosio2005}. The proof of the representation formula \eqref{equation_representation_A_sup} can easily be obtained by adapting the proof of \cite[Theorem 4.1.6]{Ambrosio2003}.  
\end{proof}

The quantity $|\dot{\rho}_t|$ is called the metric derivative of the curve $\rho$ and heuristically corresponds to the norm of the derivative of $\rho$ at time $t$ in the metric space $(\Prob(\Omega), W_2)$. The link between this metric derivative and the continuity equation is the following (and difficult) theorem, whose proof can be found in \cite[Theorem 8.3.1]{Ambrosio2005} (see also \cite[Theorem 5.14]{SantambrogioOTAM}).  

\begin{theo}
\label{theorem_equivalence_A_WMD}
Let $\rho \in \Gamma$ be a $2$-absolutely continuous curve. Then 
\begin{equation}
\label{equation_kinetic_energy_metric_derivative}
\frac{1}{2} \int_0^1 |\dot{\rho}_t|^2 \ddr t = \min \left\{ \int_0^1 \left( \int_\Omega \frac{1}{2} |\vbf_t|^2 \ddr \rho_t \right) \ddr t \right\},
\end{equation}
where the minimum is taken over all families $(\vbf_t)_{t \in \intervalleff{0}{1}}$ such that $\vbf_t \in L^2(\Omega, \R^d, \rho_t)$ for a.e. $t$ and such that the continuity equation $\dr_t \rho_t + \nabla \cdot (\rho_t \vbf_t) = 0$ with no-flux boundary conditions is satisfied in a weak sense. 
\end{theo}    
   
This result shows that if we are only interested in the minimal value taken by the action of the curve (the r.h.s. (right hand side) of \eqref{equation_kinetic_energy_metric_derivative}), we only need to consider the metric derivative of the curve $\rho$ and we can forget the velocity field $\vbf$. Therefore we define the action $A : \Gamma \to \intervalleff{0}{+ \infty}$ by, for any $\rho \in \Gamma$,  
\begin{equation}
\label{equation_defintion_A}
A(\rho) := \begin{cases}
\dst{ \frac{1}{2} \int_0^1 |\dot{\rho}_t|^2 \ddr t} & \text{if } \rho \text{ is } 2-\text{absolutely continuous}, \\
+ \infty & \text{else}.
\end{cases}
\end{equation} 
Some standard but useful properties of this functional are the following. 

\begin{prop}
\label{proposition_lsc_compact_A}
The functional $A$ is convex, l.s.c. and its sublevel sets are compact in $\Gamma$. 
\end{prop}   

\begin{proof}
To prove that $A$ is convex and l.s.c., we rely on the representation formula \eqref{equation_representation_A_sup} which shows that $A$ is the supremum of convex continuous functions. Moreover if $\rho \in \Gamma$ is a curve with finite action and $s < t$, then, again with \eqref{equation_representation_A_sup}, one can see that $W_2(\rho_s, \rho_t) \leqslant \sqrt{2 A(\rho)} \sqrt{t-s}$. This shows that the sublevel sets of $A$ are uniformly equicontinuous, therefore they are relatively compact thanks to Ascoli-Arzela's theorem. 
\end{proof}

\subsection{The heat equation and the Wasserstein space}

Let us denote by $\Phi : \intervallefo{0}{+ \infty} \times \Prob(\Omega) \to \Prob(\Omega)$ the flow of the heat equation with Neumann boundary conditions. In other words, for any $s \geqslant 0$ and any $\mu \in \Prob(\Omega)$, $\Phi_s(\mu) = u(s)$ where $u$ is the solution (in the sense of distributions) of the Cauchy problem 
\begin{equation*}
\begin{cases}
\dst{\frac{\partial u}{\partial t} = \Delta u} & \text{in } \intervalleoo{0}{+ \infty} \times \mathring{\Omega} \\
\nabla u  \cdot n = 0 & \text{on } \intervalleoo{0}{+ \infty} \times \partial \Omega \\
\dst{\lim_{t \to 0} u(t)} = \mu & \text{in } \Prob(\Omega)
\end{cases}.
\end{equation*}
In the equation above, $n$ stands for the outward normal vector to the boundary $\partial \Omega$. As $\Omega$ is convex, it has a Lipschitz boundary, a regularity which is known to be sufficient for this Cauchy problem to be well posed and to admit a unique solution (see for instance \cite[Section 7]{Arendt2002} and \cite{Pierre1982}). Moreover (see \cite[Section 7]{Arendt2002}), a regularizing effect of the heat flow is encoded in the following estimate (with $C$ a constant that depends only on $\Omega$): 
\begin{equation*} 
\forall \mu \in \Prob(\Omega), \ \forall s > 0, \ \| \Phi_s \mu \|_{L^\infty} \leqslant C \left( s^{-d/2} + 1 \right).
\end{equation*}
In particular, for any $s> 0$ there exists a constant $C_s$ such that for any $\mu \in \Prob(\Omega)$, we have $\Hen(\Phi_s \mu) \leqslant C_s$. 

The key point in what follows is that the heat flow can be seen as the gradient flow of the entropy functional $\Hen$ in the metric space $(\Prob(\Omega), W_2)$. That is (and this was remarked first by \cite{Jordan1998}), in a very informal way, $\Phi$ flows in the direction where the entropy $\Hen$ decreases the most. The standard reference about gradient flows in metric spaces is \cite{Ambrosio2005}, one can also look at the survey \cite{Santambrogio2016}. In any case, this seminal point of view explains the three following identities involving the heat flow, the Wasserstein distance, and the entropy. 

\begin{prop}
The Wasserstein distance decreases along the heat flow: if $\mu$ and $\nu \in \Prob(\Omega)$, and $s \geqslant 0$, 
\begin{equation}
\label{equation_heat_flow_decreases_W}
W_2(\Phi_s \mu, \Phi_s \nu) \leqslant W_2(\mu, \nu).
\end{equation} 
Moreover, let $\mu \in \Prob(\Omega)$ with $\Hen(\mu) < + \infty$. Then the curve $s \mapsto \Phi_s \mu$ is $2$-absolutely continuous and the Energy Identity holds: for any $s \geqslant 0$, 
\begin{equation}
\tag{EI}
\label{equation_EI}
\int_0^s |\dot{\Phi_r \mu}|^2 \ddr r = \Hen(\mu) - \Hen(\Phi_s(\mu)). 
\end{equation}
In addition, for any $\mu, \nu \in \Prob(\Omega)$ with $\Hen(\mu) < + \infty$ and any $s \geqslant 0$, the Evolution Variational Inequality holds: 
\begin{equation}
\label{equation_EVI}
\tag{EVI}
\limsup_{h \to 0, \ h>0} \frac{W_2^2(\Phi_{s+h} \mu , \nu) - W_2^2(\Phi_s \mu , \nu)}{2h} \leqslant \Hen(\nu) - \Hen(\Phi_s \mu).
\end{equation}  
\end{prop}

\noindent One can look at \cite[Theorem 11.2.1]{Ambrosio2005} to see the generalization of these identities to more general functionals than $\Hen$ along their gradient flow, provided that assumptions of $\lambda$-convexity along generalized geodesics are satisfied by the functional (and this is the case for $\Hen$ with $\lambda = 0$ because of our assumption of convexity of $\Omega$). All this properties are also summarized in \cite[Theorem 2.4]{Matthes2009}.

\subsection{Statement of the continuous problem}

As explained in the introduction, the object on which we will work, a "$W_2$-traffic plan", is a probability measure on the set of curves valued in $\Prob(\Omega)$, i.e. an element of $\Prob(\Gamma)$. Recall that the space $\Prob(\Gamma)$ is equipped with the topology of weak convergence of measures. If $Q \in \Prob(\Gamma)$, we need to translate the constraints, namely the fact that the values of the curves at $t=0$ and $t=1$ are fixed, and the incompressibility at each time $t$.

Incompressibility means that at each time $t$, the measure $e_t \# Q$ (which is an element of $\Prob(\Prob(\Omega))$) when averaged (its mean value is an element of $\Prob(\Omega)$), is equal to $\Leb$. We therefore need to define what the mean value of $e_t \# Q$ is.

\begin{defi}
\label{definition_m_t}
Let $S$ be a closed subset of $\intervalleff{0}{1}$ and $t \in S$. If $Q \in \Prob(\Gamma_S)$, we denote by $m_t(Q)$ the probability measure on $\Omega$ defined by
\begin{equation}
\label{equation_definition_m_t}
\forall a \in C(\Omega), \ \int_\Omega a(x) \ddr[m_t(Q)](x) := \int_{\Gamma_S} \left( \int_\Omega a(x) \ddr \rho_t(x) \right) \ddr Q(\rho). 
\end{equation}  
\end{defi}

\noindent We can easily see that, for a fixed $t$, $Q \mapsto m_t(Q)$ is continuous. It is an easy application of Fubini's theorem to show that, if $Q$-a.e. $\rho_t$ is absolutely continuous w.r.t. to $\Leb$, then $m_t(Q)$ is also absolutely continuous w.r.t. $\Leb$, and its density is the mean density of the $\rho_t$ w.r.t. $Q$. Incompressibility is then expressed by the fact that $m_t(Q) = \Leb$ for any $t$. 

To encode the boundary conditions, we just consider a coupling $\gamma \in \Prob(\Gamma_{\{ 0,1 \}}) = \Prob( \Prob(\Omega) \times \Prob(\Omega) )$ between the initial and final values, compatible with the incompressibilty constraint (i.e. $m_0(\gamma) = m_1(\gamma) = \Leb$), and we impose that $(e_0, e_1) \# Q = \gamma$.  

\begin{defi}
\label{definition_admissible_traffic_plans}
Let $\gamma \in \Prob(\Gamma_{\{ 0,1 \}})$ be a coupling compatible with the incompressibility constraint (i.e. $m_0(\gamma) = m_1(\gamma) = \Leb$) and $S$ be a closed subset of $\intervalleff{0}{1}$ containing $0$ and $1$. The space of incompressible $W_2$-traffic plans is 
\begin{equation*}
\Probin(\Gamma_S) := \left\{ Q \in \Prob(\Gamma_S)\ :  \ \forall t \in S, \ m_t(Q) = \Leb  \right\}.
\end{equation*}
The space of $W_2$-traffic plans satisfying the boundary conditions is 
\begin{equation*}
\Probbc(\Gamma_S) := \left\{ Q \in \Prob(\Gamma_S)\ :  \ (e_0, e_1) \# Q = \gamma  \right\}.
\end{equation*}
The space of admissible $W_2$-traffic plans is 
\begin{equation*}
\Proba(\Gamma_S) := \Probin(\Gamma_S) \cap \Probbc(\Gamma_S).
\end{equation*}
\end{defi}    

\noindent The following proposition derives directly from the definition.

\begin{prop}
\label{proposition_Padmissible_closed}
If $S$ is a closed subset of $\intervalleff{0}{1}$ containing $0$ and $1$, the spaces $\Probin(\Gamma_S)$, $\Probbc(\Gamma_S)$ and $\Proba(\Gamma_S)$ are closed in $\Prob(\Gamma_S)$.
\end{prop}

We have now enough vocabulary to state the minimization problem we are interested in, namely to minimize the averaged action over the set of admissible $W_2$-traffic plans. We denote by $\A : \Prob(\Gamma) \to \intervalleff{0}{+ \infty}$ the functional defined by, for any $Q \in \Prob(\Gamma)$, 
\begin{equation*}
\A(Q) := \int_\Gamma A(\rho) \ddr Q(\rho),
\end{equation*} 
where we recall that $A(\rho)$ is the action of the curve $\rho$, see \eqref{equation_defintion_A}.

\begin{defi}
The continuous problem is defined as
\begin{equation}
\label{equation_continuous_problem}
\tag{CP}
\min \{ \A(Q)\ :  \ Q \in \Proba(\Gamma) \}.
\end{equation}
Any $Q \in \Proba(\Gamma)$ with $\A(Q) < + \infty$ realizing the minimum will be referred as a solution of the continuous problem.
\end{defi}

In order to prove the existence of a solution to \eqref{equation_continuous_problem}, we rely on the classical following lemma which is valid if $\Gamma_S$ is replaced by any metric space (see for instance \cite[Proposition 7.1]{SantambrogioOTAM} and \cite[Remark 5.15]{Ambrosio2005}).  

\begin{lm}
\label{lemma_lsc_compact_averaged}
Let $S$ be a closed subset of $\intervalleff{0}{1}$ and $F : \Gamma_S \to \intervalleff{0}{+ \infty}$ a l.s.c. positive function. Then the function $\F : \Prob(\Gamma_S) \to \intervalleff{0}{+ \infty}$ defined by
\begin{equation*}
\F(Q) = \int_{\Gamma_S} F(\rho) \ddr Q(\rho)
\end{equation*}
is convex and l.s.c. Moreover, if the sublevel sets of $F$ are compact, so are those of $\F$. 
\end{lm}

\noindent The existence of a solution to \eqref{equation_continuous_problem} is then a straightforward application of the direct method of calculus of variations. 

\begin{theo}
\label{theorem_existence_solution_CP}
There exists at least one solution to \eqref{equation_continuous_problem}.
\end{theo}

\begin{proof}
The functional $\A$ is l.s.c. and has compact sublevel sets thanks to Proposition \ref{proposition_lsc_compact_A} and Lemma \ref{lemma_lsc_compact_averaged}. Moreover the set $\Proba(\Gamma)$ is closed. To use the direct method of calculus of variations, we only need to prove that there exists $Q \in \Proba(\Gamma)$ such that $\A(Q) < + \infty$. 

Notice that as $\Omega$ is convex, it is the image of the unit cube of $\R^d$ by a Lipschitz and measure-preserving map (see \cite[Theorem 5.4]{Fonseca1992}\footnotemark). It is known (see \cite[Theorem 3.3]{Ambrosio2009} and Proposition \ref{proposition_parametric_to_nonparametric} to translate the result in our setting) that the fact that $\Omega$ is the image of the unit cube by a Lipschitz and measure-preserving map ensures the existence of an admissible $W_2$-traffic plan with finite action.   
\end{proof}

\footnotetext{Strictly speaking, in \cite{Fonseca1992}, it is required that $\Omega$ has a piecewise $C^1$ boundary, but this assumption is only used to prove that the Minkowski functional of $\Omega$ is Lipschitz. If $\Omega$ is convex, then its Minkowski functional is convex, hence Lipschitz. Thus, one can drop the assumption of a piecewise $C^1$ boundary if $\Omega$ is convex.}

Let us mention here some already known results about the continuous problem (the existence of a solution being one of them). In some cases there is no uniqueness in \eqref{equation_continuous_problem}, we refer the reader to \cite{Bernot2009} for a comprehensive study of one of such cases. In \cite{Ambrosio2009}, it is shown how, from a pair of measure-preserving plans (i.e. a pair of elements of $\{ \mu \in \Prob(\Omega \times \Omega) \ : \pi_0 \# \mu = \Leb \text{ and } \pi_1 \# \mu = \Leb \}$), one can build an incompressible coupling $\gamma \in\Probin(\Gamma_{\{ 0,1\}})$ which is, in some sense, a concatenation of them. Indeed, if $\mu, \nu \in \Prob(\Omega \times \Omega)$ are two measure-preserving plans, one can consider $(\mu_x)_{x \in \Omega}$ and $(\nu_x)_{x \in \Omega}$ the disintegration of $\mu$ and $\nu$ w.r.t. $\pi_1$, and then define $\gamma \in \Probin(\Gamma_{\{ 0,1 \}})$ by its action on continuous functions $a \in C(\Gamma_{\{ 0,1 \}}) = C(\Prob(\Omega)^2)$: 
\begin{equation*}
\int_{\Gamma_{\{ 0,1 \}}} a(\rho_0, \rho_1) \ddr \gamma(\rho_0, \rho_1) := \int_\Omega a(\mu_x, \nu_x) \ddr \Leb(x). 
\end{equation*}  
(To understand this construction and check that it is incompressible, one can look at \cite{Ambrosio2009} and Section \ref{section_equivalence_formulation} for the translation in the non parametric case, as \cite{Ambrosio2009} corresponds to the parametric case with $(\Afr, \theta) = (\Omega, \Leb)$). Using this construction, then \cite[Proposition 3.4]{Ambrosio2009} states that the minimal cost (or more precisely the square root of the minimal value of \eqref{equation_continuous_problem}) defines a distance on the space of measure-preserving plans.   

\bigskip

In this article we are interested in the temporal behavior of the entropy when averaged over all phases.

\begin{defi}
Let $S$ be a closed subset of $\intervalleff{0}{1}$. For any $Q \in \Prob(\Gamma_S)$, we define the \emph{averaged} entropy $H_Q : S \to \intervalleff{0}{+ \infty}$ by, for any $t \in S$,
\begin{equation*}
H_Q(t) := \int_{\Gamma_S} \Hen(\rho_t) \ddr Q(\rho).
\end{equation*}
If $Q \in \Prob(\Gamma)$, the quantity $\dst{\int_0^1 H_Q(t) \ddr t}$ will be called the \emph{total} entropy of $Q$.
\end{defi}

\noindent We recall that $\Hen$ is the entropy of a probability measure w.r.t. $\Leb$, see \eqref{definition_entropy}. By lower semi-continuity of $\Hen$ and Lemma \ref{lemma_lsc_compact_averaged}, we can see that the function (of the variable $t$) $H_Q$ is l.s.c. In the sequel, we will concentrate on the cases where the averaged entropy belongs to $L^1(\intervalleff{0}{1})$, i.e. where the total entropy is finite. By doing so, we exclude classical solutions: indeed, for a classical solution $Q \in \Proba(\Gamma)$, for any $t$ the measure $e_t \# Q$ is concentrated on Dirac masses, for which the entropy is infinite. We denote by $\Proba^H(\Gamma)$ the set of admissible $W_2$-traffic plans for which the total entropy is finite: 
\begin{equation*}
\Proba^H(\Gamma) := \Proba(\Gamma) \cap \left\{ Q \in \Prob(\Gamma) \ :  \ \int_0^1 H_Q(t) \ddr t < + \infty \right\}
\end{equation*} 
The main (and restrictive) assumption that we will consider is that there exists a solution of the continuous problem \eqref{equation_continuous_problem} in $\Proba^H(\Gamma)$: 

\begin{asmp}
\label{assumption_L1_entropy}
There exists $Q \in \Proba^H(\Gamma)$ such that $\A(Q) = \min \{ \A(Q')\ :  \ Q' \in \Proba(\Gamma) \}$.
\end{asmp}
\noindent We will also work with a second assumption which will turn out to be more restrictive than Assumption \ref{assumption_L1_entropy}, but which has the advantage of involving only the boundary terms, namely the fact that the initial and final values have finite averaged entropy. 
\begin{asmp}
\label{assumption_finite_entropy_IF}
The coupling $\gamma$ is such that $H_\gamma(0)$ and $H_\gamma(1)$ are finite.
\end{asmp}  
\noindent In other words, we impose that 
\begin{equation*}
\int_\gamma \left( \int_\Omega \rho_0 \ln \rho_0 \right) \ddr \gamma(\rho) < + \infty \ \text{ and } \ \int_\gamma \left( \int_\Omega \rho_1 \ln \rho_1 \right) \ddr \gamma(\rho) < + \infty.
\end{equation*}
\noindent In particular, Assumption \ref{assumption_finite_entropy_IF} implies that $e_0 \# \gamma$ and $e_1 \# \gamma$ are concentrated on measures that are absolutely continuous w.r.t. $\Leb$: it excludes any classical boundary data.  

\bigskip

The two main results of this paper can be stated as follows. Recall that $\Omega$ is assumed to be convex.
\begin{theo}
\label{theorem_finite_entropy_IF_implies_L1}
Suppose that Assumption \ref{assumption_finite_entropy_IF} holds. Then there exists a solution $Q \in \Proba(\Gamma)$ of the continuous problem \eqref{equation_continuous_problem} such that $H_Q(t) \leqslant \max(H_\gamma(0), H_\gamma(1))$ for any $t \in \intervalleff{0}{1}$.
\end{theo}

\noindent In other words, if the initial and final averaged entropy are finite, then there exists a solution of the continuous problem with a uniformly bounded averaged entropy. In particular, Assumption \ref{assumption_finite_entropy_IF} implies Assumption \ref{assumption_L1_entropy}. 

\begin{theo}
\label{theorem_main_convexity_entropy}
Suppose that Assumption \ref{assumption_L1_entropy} holds. Then, among all the solutions of the continuous problem \eqref{equation_continuous_problem}, the unique $Q \in \Proba^H(\Gamma)$ which minimizes the total entropy $\int_0^1 H_Q(t) \ddr t$ is such that $H_Q$ is convex.   
\end{theo}

\noindent In other words, we are able to prove the convexity of the averaged entropy for the solution which is "the most mixed", i.e. the one for which the total entropy is minimal. This statement contains the fact that the criterion of minimization of the total entropy selects a unique solution among the -- potentially infinitely many -- solutions of \eqref{equation_continuous_problem}. Let us mention that our proof could be easily adapted to show that the convexity also holds for the solution $Q$ which minimizes $\int_0^1 H_Q(t) a(t) \ddr t$, where $a : \intervalleff{0}{1} \to \intervalleoo{0}{+ \infty}$ is any continuous and strictly positive function.  
 
The next two sections are devoted to the proof of these two theorems. As explained in the introduction, we will introduce a discrete (in time) problem \eqref{equation_discrete_problem} which approximates the continuous one. Without any assumption, we will be able to prove the convexity of the averaged entropy at the discrete level (Theorem \ref{theorem_convexity_entropy_discrete}). Then we will show that, under Assumption \ref{assumption_L1_entropy} or Assumption \ref{assumption_finite_entropy_IF}, the solutions of the discrete problems converge to a solution of the continuous one (Proposition \ref{proposition_optimality_barQ}). Under Assumption \ref{assumption_finite_entropy_IF}, this solution will happen to have a uniformly bounded entropy (Corollary \ref{corollary_finite_entropy_IF_implies_bounded}). Then we will show that, under Assumption \ref{assumption_L1_entropy}, this solution will be the one with minimal total entropy (Corollary \ref{proposition_barQ_lowest_entropy}) and that its averaged entropy is a convex function of time (Corollary \ref{corollary_barQ_true_convex_entropy}). 

Finally, the uniqueness of such a $Q \in \Proba^H(\Gamma)$ with minimal total entropy has nothing to do with the discrete problem, it is a simple consequence of the strict convexity of $\Hen$. We will therefore prove it here to end this section. Indeed, it is a consequence of the following proposition.

\begin{prop}
\label{proposition_strict_convexity_entropy}
Let $Q^1$ and $Q^2 \in \Proba^H(\Gamma)$ be two distinct admissible $W_2$-traffic plans. Then there exists $Q \in \Proba^H(Q)$ with 
\begin{equation*}
\A(Q) \leqslant \frac{1}{2} \left( \A(Q^1) + \A(Q^2) \right)
\end{equation*} 
and 
\begin{equation*}
\int_0^1 H_Q(t) \ddr t < \frac{1}{2} \left( \int_0^1 H_{Q^1}(t) \ddr t + \int_0^1 H_{Q^2}(t) \ddr t \right).
\end{equation*}
\end{prop} 

\begin{proof}
As $Q \mapsto H_Q$ is linear, it is not sufficient to consider the mean of $Q^1$ and $Q^2$. Instead, we will need to take means in $\Gamma$. In order to do so, we disintegrate $Q^1$ and $Q^2$ w.r.t. $e_{\{ 0,1 \}} = (e_0, e_1)$. We obtain two families $Q^1_{\rho_0, \rho_1}$ and $Q^2_{\rho_0, \rho_1}$ of $W_2$-traffic plans indexed by $(\rho_0, \rho_1) \in \Gamma_{\{ 0, 1 \}} = \Prob(\Omega)^2$. We define $Q$ by its disintegration w.r.t. $e_{\{ 0,1 \}}$: we set $Q := Q_{\rho_0, \rho_1} \otimes \gamma$ where $Q_{\rho_0, \rho_1}$ is taken to be the image measure of $Q^1_{\rho_0, \rho_1} \otimes Q^2_{\rho_0, \rho_1}$ by the map $(\rho^1, \rho^2) \mapsto (\rho^1 + \rho^2)/2$ (where the $+$ refers to the usual affine structure on $\Gamma$). In other words, for any $a \in C(\Gamma)$, 
\begin{equation*}
\int_\Gamma a(\rho) \ddr Q(\rho) := \int_{\Gamma_{\{ 0,1 \}}} \left( \int_\Gamma a \left[ \frac{\rho^1 + \rho^2}{2} \right] \ddr Q^1_{\rho_0, \rho_1}(\rho^1) \ddr Q^2_{\rho_0, \rho_1}(\rho^2) \right) \ddr \gamma(\rho_0, \rho_1).
\end{equation*}
As $(e_0,e_1) \# Q^1_{\rho_0, \rho_1}$ and $(e_0,e_1) \# Q^2_{\rho_0, \rho_1}$ are Dirac masses concentrated on $(\rho_0, \rho_1)$, we can easily see that $Q \in \Probbc(\Gamma)$. The incompressibility constraint is straightforward to obtain: for any $a \in C(\Omega)$ and any $t \in \intervalleff{0}{1}$, 
\begin{align*}
\int_\Omega a(x) \ddr [m_t(Q)](x)&  = \int_{\Gamma_{\{ 0,1 \}}} \left( \int_\Gamma \left[ \int_\Omega  a(x) \frac{\ddr \rho^1_t(x) + \ddr \rho^2_t(x)}{2} \right] \ddr Q^1_{\rho_0, \rho_1}(\rho^1) \ddr Q^2_{\rho_0, \rho_1}(\rho^2) \right) \ddr \gamma(\rho_0, \rho_1) \\
& =\int_{\Gamma_{\{ 0,1 \}}} \left( \int_\Gamma \left[ \int_\Omega  a(x) \frac{\ddr \rho^1_t(x)}{2} \right] \ddr Q^1_{\rho_0, \rho_1}(\rho^1) + \int_\Gamma \left[ \int_\Omega  a(x) \frac{\ddr \rho^2_t(x)}{2} \right] \ddr Q^2_{\rho_0, \rho_1}(\rho^2) \right) \ddr \gamma(\rho_0, \rho_1) \\
& = \frac{1}{2} \int_\Omega a(x) \ddr x + \frac{1}{2} \int_\Omega a(x) \ddr x \\
& = \int_\Omega a(x) \ddr x. 
\end{align*}  
Thus, we have $Q \in \Proba(\Gamma)$. To handle the action, let us just remark that for any $\rho^1$ and $\rho^2$ in $\Gamma$, by convexity of $A$, 
\begin{equation*}
A \left( \frac{\rho^1 + \rho^2}{2} \right) \leqslant \frac{1}{2} \left( A(\rho^1) + A(\rho^2) \right).
\end{equation*}
Integrating this inequality w.r.t. to $Q^1_{\rho_0, \rho_1} \otimes Q^2_{\rho_0, \rho_1}$ and then w.r.t. $\gamma$ gives the result. We use the same kind of reasoning for the entropy, but this functional is strictly convex. Hence, for any $t \in \intervalleff{0}{1}$, 
\begin{equation*}
\Hen \left( \frac{\rho^1_t + \rho^2_t}{2} \right) \leqslant \frac{1}{2} \left( \Hen(\rho^1_t) + \Hen(\rho^2_t) \right)
\end{equation*}
with a strict inequality if $\rho^1_t \neq \rho^2_t$ and if the r.h.s. is finite. Integrating w.r.t. $t$ and w.r.t. $Q^1_{\rho_0, \rho_1} \otimes Q^2_{\rho_0, \rho_1}$ we get, 
\begin{multline*}
\int_\Gamma \left( \int_0^1 \Hen \left[ \frac{\rho^1_t + \rho^2_t}{2} \right] \ddr t \right) \ddr Q^1_{\rho_0, \rho_1}(\rho^1) \ddr Q^2_{\rho_0, \rho_1}(\rho^2) \leqslant \\
\frac{1}{2} \left(  \int_\Gamma \left( \int_0^1 \Hen [\rho^1_t] \ddr t \right) \ddr Q^1_{\rho_0, \rho_1}(\rho^1) + \int_\Gamma \left( \int_0^1 \Hen [\rho^2_t] \ddr t \right) \ddr Q^2_{\rho_0, \rho_1}(\rho^2) \right),
\end{multline*} 
with a strict inequality if $Q^1_{\rho_0, \rho_1} \neq Q^2_{\rho_0, \rho_1}$ and if the r.h.s. is finite. Then, we integrate w.r.t. $\gamma$ and notice that, as $Q^1 \neq Q^2$, then $Q^1_{\rho_0, \rho_1} \neq Q^2_{\rho_0, \rho_1}$ for a $\gamma$-non negligible sets of $(\rho_0, \rho_1)$, and as $Q^1$ and $Q^2 \in \Proba^H(\Gamma)$, the r.h.s. of the equation above is finite for $\gamma$-a.e. $(\rho_0, \rho_1)$. Using Fubini's theorem, we are led to the announced conclusion.
\end{proof}

\section{Analysis of the discrete problem}
\label{section_discrete problem}

As we explained before, to tackle the continuous problem \eqref{equation_continuous_problem}, we will introduce a discretized (in time) variational problem that approximates the continuous one. In this section, we give a brief heuristic justification of it, prove its well-posedness, and show that the discrete averaged entropy is convex. In the proof of the latter property, we use the \emph{flow interchange} technique that was previously introduced in \cite{Matthes2009}.

The discrete problem is obtained by performing three different approximations: 
\begin{itemize}
\item[•] We consider a number of discrete times $N +1 \geqslant 2$. We will use $\tau := 1/N$ as a notation for the time step. The set $T^N \subset \intervalleff{0}{1}$ will stand for the set of all discrete times, namely
\begin{equation*}
T^N := \left\{ k \tau \ :  \ k = 0,1, \ldots, N \right\}.
\end{equation*}
In particular, $\Gamma_{T^N} = \Prob(\Omega)^{N+1}$. We will work with $W_2$-traffic plans on $\Gamma_{T^N}$, i.e. elements of $\Prob(\Gamma_{T^N})$. According to the representation of the action \eqref{equation_representation_A_sup}, we expect that for a curve $\rho \in \Gamma$, 
\begin{equation*}
A(\rho) \simeq \sum_{k=1}^N \frac{W_2^2(\rho_{(k-1) \tau}, \rho_{k \tau})}{2 \tau}. 
\end{equation*}  
\item[•] The incompressibility constraint will be relaxed. For any $k \in \{ 1, 2, \ldots, N-1 \}$, we penalize the densities $m_{k \tau}(Q)$ which are away from the Lebesgue measure by adding a term $\C_q(m_{k \tau}(Q))$, where $\C_q$ is defined in \eqref{equation_definition_power}. As explained in Section \ref{section_preliminaries}, this term is positive and vanishes if and only if $m_{k \tau}(Q) = \Leb$, moreover it goes to $+ \infty$ as $q \to + \infty$ if $m_{k \tau}(Q) \neq \Leb$.     
\item[•] We will also add an entropic penalization, i.e. a discretized version of 
\begin{equation*}
\lambda \int_0^1 H_Q(t) \ddr t,
\end{equation*} 
with $\lambda$ a small parameter. This term explains why we select, at the limit $\lambda \to 0$, the minimizers whose total entropy is minimal. It is crucial because it enables us to show that the averaged entropy of the discrete problem converges pointwisely to the averaged entropy of the continuous problem. This pointwise convergence is necessary to ensure that the averaged entropy of the continuous problem is convex. In particular, the limit $\lambda \to 0$ must be taken after $N \to + \infty$ and $q \to + \infty$.   
\end{itemize}  

Let us state formally our discrete minimization problem. We fix $N \geqslant 1$ ($\tau := 1 / N$), $q > 1$ and $\lambda > 0$ and define $T^N = \{ k \tau \ : \ k = 0, 1, \ldots, N \}$. We denote by $\A^{N, q, \lambda} : \Prob(\Gamma_{T^N}) \to \intervalleff{0}{+\infty}$ the functional defined by, for any $Q \in \Prob(\Gamma_{T^N})$, 
\begin{equation*}
\A^{N, q, \lambda} (Q) := \sum_{k=1}^N \int_{\Gamma_{T^N}} \frac{W_2^2(\rho_{(k-1) \tau}, \rho_{k \tau})}{2 \tau} \ddr Q(\rho) + \sum_{k=1}^{N-1} \C_q(m_{k \tau}(Q)) + \lambda \sum_{k=1}^{N-1} \tau H_{Q}\left( k \tau \right).
\end{equation*}
The \emph{Discrete Problem} consists in minimizing this functional under the only constraint that the initial and final values are coupled through $\gamma$, the set of such $W_2$-traffic plans being $\Probbc(\Gamma_{T^N})$ (cf. Definition \ref{definition_admissible_traffic_plans}): 
\begin{equation}
\tag{DP}
\label{equation_discrete_problem}
\min \left \{ \A^{N,q, \lambda}(Q)\ :  \ Q \in \Probbc(\Gamma_{T^N}) \right\}.
\end{equation} 
A solution of the discrete problem is a $Q \in \Probbc(\Gamma_{T^N})$ with $\A^{N,q, \lambda}(Q) < + \infty$ which minimizes $\A^{N,q, \lambda}$.

\begin{prop}
The discrete problem \eqref{equation_discrete_problem} admits a solution. 
\end{prop} 

\begin{proof}
We can see that $\A^{N,q, \lambda}$ is a positive l.s.c. functional. Lower semi-continuity of the discretized action and of the entropic penalization are not difficult to see thanks to Lemma \ref{lemma_lsc_compact_averaged}. Moreover, $Q \mapsto \C_q(m_t(Q))$ is the composition of the linear and continuous map $Q \mapsto m_t(Q)$ and of the l.s.c. map $\C_q$, hence is l.s.c.

As the space $\Prob(\Gamma_{T^N}) = \Prob( \Prob(\Omega)^{N+1} )$ is compact, $\Probbc(\Gamma_{T^N})$ is also a compact space, thus it is enough to show that there exists one $Q \in \Probbc(\Gamma_{T^N})$ such that $\A^{N, q, \lambda}(Q) < + \infty$. We take $Q$ to be equal to $\gamma$ on the endpoints, and such that $e_{k \tau} \# Q$ is a Dirac mass concentrated on the Lebesgue measure $\Leb$ for any $k \in \{ 1, 2, \ldots, N-1 \}$. As $\Hen(\Leb) = 0$ and as the incompressibility constraint $m_{k \tau}(Q) = \Leb$ is satisfied for every $k \in \{ 0, 1, \ldots, N \}$, we can see that for this $Q$ we have 
\begin{equation*}
\A^{N, q, \lambda}(Q) = \int_{\Gamma_{\{ 0,1 \}}} \frac{W_2^2(\rho_0, \Leb) + W_2^2(\Leb, \rho_1)}{2 \tau} \ddr \gamma (\rho). 
\end{equation*}  
As the Wasserstein distance is uniformly bounded by the diameter of $\Omega$, the r.h.s. of the above equation is finite. The conclusion derives from a straightforward application of the direct method of calculus of variations. 
\end{proof}

One could show that the discrete problem \eqref{equation_discrete_problem} admits a unique solution (it is basically the same proof as Proposition \ref{proposition_strict_convexity_entropy}), but we will not need it. The key result of this section is the following:

\begin{theo}
\label{theorem_convexity_entropy_discrete}
Let $Q \in \Probbc(\Gamma_{T^N})$ be a solution of the discrete problem \eqref{equation_discrete_problem}. Then the function $k \in \{ 0, 1,\ldots, N \} \mapsto H_Q(k \tau)$ is convex, i.e. for every $k \in \{ 1, 2, \ldots, N-1 \}$, 
\begin{equation}
\label{equation_convexity_entropy_discrete}
H_Q \left( k \tau \right) \leqslant \frac{1}{2} H_Q \left( (k-1) \tau \right) + \frac{1}{2} H_Q \left( (k+1) \tau \right).
\end{equation} 
\end{theo}

\begin{proof}
As $\A^{N, q, \lambda}(Q)$ is finite we know that for every $k \in \{ 1,2, \ldots, N-1 \}$, $H_Q(k \tau) < + \infty$ and $m_{k \tau}(Q) \in L^q(\Omega)$. Let us remark that if $H_Q(0) = + \infty$ then there is nothing to prove in equality \eqref{equation_convexity_entropy_discrete} for $k=1$ (the r.h.s. being infinite); and, equivalently, if $H_Q(1) = + \infty$ there is nothing to prove for $k=N-1$. So from now on, we fix $k \in \{ 1, 2, \ldots, N-1 \}$ such that $\dst{H_Q \left( (k-1) \tau \right)}$, $\dst{H_Q \left( k \tau \right)}$ and $\dst{H_Q \left( (k+1) \tau \right)}$ are finite, and it is enough to show \eqref{equation_convexity_entropy_discrete} for such a $k$.

We recall that $\Phi : \intervallefo{0}{+ \infty} \times \Prob(\Omega) \to \Prob(\Omega)$ denotes the heat flow, let us call $\Phi^k : \intervallefo{0}{+ \infty} \times \Gamma_{T^N} \to \Gamma_{T^N}$ the heat flow acting only on the $k$-th component: for any $s \geqslant 0$, $\rho \in \Gamma_{T^N}$ and $l \in \{ 0, 1, \ldots, N \}$, 
\begin{equation*}
\Phi^k_s(\rho)(l \tau) := \begin{cases}
\Phi_s (\rho_{l \tau}) & \text{if } l = k, \\
\rho_{l \tau}& \text{if } l \neq k.
\end{cases}
\end{equation*}
If $s \geqslant 0$, it is clear that $\Phi^k_s$ leaves unchanged the boundary values, thus $\Phi_s^k \# Q \in \Probbc(\Gamma_{T^N})$, and therefore by optimality of $Q$ we have that 
\begin{equation}
\label{equation_Q_better_Q_along_flow}
\A^{N, q, \lambda}(Q) \leqslant \A^{N,q, \lambda}(\Phi^k_s \# Q).
\end{equation}  
Let us expand this formula. We can see (by definition of $H_Q$) that
\begin{equation*}
H_{\Phi_s^k \# Q} \left( l \tau \right) = \begin{cases}
\dst{\int_{\Gamma_{T^N}} \Hen(\Phi_s [\rho_{l \tau}]) \ddr Q (\rho)} & \text{if } l = k, \\
\dst{H_Q \left( l \tau \right)} & \text{if } l \neq k.
\end{cases}
\end{equation*}
Concerning the term $m_{l \tau}(Q)$, the linearity of the flow enables us to write
\begin{equation*}
m_{l \tau}(\Phi^k_s \# Q) = \begin{cases}
\Phi_s \left( m_{l \tau}[Q] \right) & \text{if } l = k, \\ 
m_{l \tau}(Q) & \text{if } l \neq k.
\end{cases}
\end{equation*}
Let us underline that the linearity of the heat flow is crucial to handle the congestion term. Our proof would not have worked if we would have wanted to show the convexity (w.r.t. time) of a functional (different from the entropy) whose gradient flow in the Wasserstein space were not linear. We can rewrite \eqref{equation_Q_better_Q_along_flow} in the following form (all the terms that do not involve the time $k \tau$ cancel): 
\begin{multline*}
\int_{\Gamma_{T^N}} \frac{W_2^2(\rho_{(k-1) \tau}, \rho_{k \tau}) + W_2^2(\rho_{k \tau}, \rho_{(k+1) \tau})}{2 \tau} \ddr Q(\rho) + \C_q(m_{k \tau}(Q)) + \lambda \tau \int_{\Gamma_{T^N}} \Hen(\rho_{k \tau}) \ddr Q(\rho) \\
\leqslant \int_{\Gamma_{T^N}} \frac{W_2^2(\rho_{(k-1) \tau}, \Phi_s \rho_{k \tau}) + W_2^2(\Phi_s \rho_{k \tau}, \rho_{(k+1) \tau})}{2 \tau} \ddr Q(\rho) + \C_q(\Phi_s \left( m_{l \tau}[Q] \right)) + \lambda \tau \int_{\Gamma_{T^N}} \Hen(\Phi_s \rho_{k \tau}) \ddr Q(\rho).
\end{multline*} 
It is a well known fact that the heat flows decreases the $L^q$ norm, thus $\C_q(\Phi_s \left( m_{k \tau}[Q] \right)) \leqslant \C_q(m_{k \tau}(Q))$. It is also a well known fact the the heat flow decreases the entropy (it is for example encoded in \eqref{equation_EI}), thus 
\begin{equation*}
\int_{\Gamma_{T^N}} \Hen(\Phi_s \rho_{k \tau}) \ddr Q(\rho) \leqslant \int_{\Gamma_{T^N}} \Hen(\rho_{k \tau}) \ddr Q(\rho). 
\end{equation*}
Therefore, multiplying by $\tau$ and dividing by $s$, we are left with the following inequality, valid for any $s > 0$: 
\begin{equation*}
\int_{\Gamma_{T^N}} \frac{W_2^2(\rho_{(k-1) \tau}, \Phi_s \rho_{k \tau}) - W_2^2(\rho_{(k-1) \tau}, \rho_{k \tau})}{2s} \ddr Q(\rho)
+ \int_{\Gamma_{T^N}} \frac{W_2^2(\Phi_s \rho_{k \tau}, \rho_{(k+1) \tau}) - W_2^2( \rho_{k \tau}, \rho_{(k+1) \tau})}{2s} \ddr Q(\rho) \geqslant 0.   
\end{equation*}  
The integrand of the first integral is exactly the rate of increase of the function $s \mapsto W_2^2(\rho_{(k-1) \tau}, \Phi_s \rho_{k \tau})/2$ whose $\limsup$ is bounded, when $s \to 0$, by $\Hen(\rho_{(k-1) \tau}) - \Hen(\rho_{k \tau})$ according to \eqref{equation_EVI}. Moreover, as the entropy is positive, the same inequality \eqref{equation_EVI} shows that this rate of increase is uniformly (in $s$) bounded from above by $\Hen(\rho_{(k-1) \tau})$, and the latter is integrable w.r.t. to $Q$. Hence by applying a reverse Fatou's lemma, we see that 
\begin{equation*}
\int_{\Gamma_{T^N}} [ \Hen(\rho_{(k-1) \tau}) - \Hen(\rho_{k \tau}) ] \ddr Q(\rho) \geqslant \limsup_{s \to 0} \int_{\Gamma_{T^N}} \frac{W_2^2(\rho_{(k-1) \tau}, \Phi_s \rho_{k \tau}) - W_2^2(\rho_{(k-1) \tau}, \rho_{k \tau})}{2s} \ddr Q(\rho).
\end{equation*}
We have a symmetric minoration for $\dst{\int_{\Gamma_{T^N}} [ \Hen(\rho_{(k+1) \tau}) - \Hen(\rho_{k \tau}) ] \ddr Q(\rho)}$, hence we end up with 
\begin{align*}
0 & \leqslant \int_{\Gamma_{T^N}} [ \Hen(\rho_{(k-1) \tau}) - \Hen(\rho_{k \tau}) ] \ddr Q(\rho) + \int_{\Gamma_{T^N}} [ \Hen(\rho_{(k+1) \tau}) - \Hen(\rho_{k \tau}) ] \ddr Q(\rho) \\
& = \int_{\Gamma_{T^N}} [ \Hen(\rho_{(k-1) \tau}) + \Hen(\rho_{(k+1) \tau}) - 2 \Hen(\rho_{k \tau}) ] \ddr Q(\rho) \\
& = H_Q \left( (k-1) \tau \right) + H_Q \left( (k+1) \tau \right) - 2 H_Q \left( k \tau \right). \qedhere
\end{align*} 
\end{proof}

\section{Limit of the discrete problems to the continuous one}
\label{section_discrete_to_continuous}

In all this section, let us denote by $\bar{Q}^{N, q, \lambda}$ a solution (in fact there exists only one but this is not important) of the discrete problem \eqref{equation_discrete_problem} with parameters $N$, $q$ and $\lambda$. We want to pass to the limit in the following way: 
\begin{itemize}
\item[•] By sending $q \to + \infty$, the incompressibility constraint $m_t(Q) = \Leb$ will be strictly enforced at the discrete times $t \in T^N$. 
\item[•] Then, we will interpolate geodesically between discrete instants and show that this builds a sequence of $W_2$-traffic plans which converges to a limit $\bar{Q}^\lambda \in \Proba(\Gamma)$ when $N \to + \infty$. This $\bar{Q}^\lambda$ is expected to be a solution 
\begin{equation*}
\min \left\{ \A(Q) + \lambda \int_0^1 H_Q(t) \ddr t\ :  \ Q \in \Proba^H(\Gamma) \right\}.
\end{equation*} 
\item[•] In the end, when $\lambda \to 0$, the $W_2$-traffic plans $\bar{Q}^\lambda$ will converge to the solution $\bar{Q}$ of the original problem with minimal total entropy and $\int_0^1 H_{\bar{Q}^\lambda}(t) \ddr t$ will converge to $\int_0^1 H_{\bar{Q}}(t) \ddr t$. This is the convergence of the total entropy that enables us to get a pointwise convergence of the averaged entropy. 
\end{itemize}
Basically, we are performing three successive $\Gamma$-limits. Let us stress out that the order in which the limits are taken is important, though this importance may be hard to see under the various technical details. In particular taking the limit $\lambda \to 0$ at the end is needed to show that at the limit the selected minimizer of the continuous problem is the one with minimal total entropy (cf. the proof of Proposition \ref{proposition_barQ_lowest_entropy}). It may be possible to take the limit $N \to + \infty$ first (instead of sending $q \to + \infty$ first), but then the incompressibility constraint must be handled differently from us.   

This section is organized as follows. First we show some kind of $\Gamma-\limsup$, i.e. given continuous curves we build discrete ones whose discrete action and total entropy are close to their continuous counterparts. Then, and thanks to these constructions, we show a uniform bound on $\bar{Q}^{N, q, \lambda}$ that allows us to extract converging subsequences toward a limit $\bar{Q}$, and we show that $\bar{Q}$ is a solution of the continuous problem. Finally, we show that $\bar{Q}$ is the minimizer of $\A$ with minimal total entropy and that its averaged entropy is convex. 

\subsection{Building discrete curves from continuous ones}

Let us first show a result that will be crucial to handle Assumption \ref{assumption_finite_entropy_IF}, namely a procedure to regularize curves in order for the total entropy to be finite.  

\begin{prop}
\label{proposition_approximation_finite_entropy}
Under Assumption \ref{assumption_finite_entropy_IF}, for any $Q \in \Proba(\Gamma)$ and for any $\varepsilon > 0$, there exists $Q' \in \Proba^H(\Gamma)$ such that $\A(Q') \leqslant \A(Q) + \varepsilon$ and $H_{Q'} \in L^\infty(\intervalleff{0}{1})$.
\end{prop}  

\begin{proof}
Let us fix $Q \in \Proba(\Gamma)$. The idea is to use the heat flow $\Phi$ to regularize the curves: indeed, we know that if $s > 0$ is fixed, then for any $\rho \in \Gamma$, $\Hen(\Phi_s \rho_t)$ is bounded independently on $t$ and $\rho$. Moreover, applying uniformly the heat flow decreases the action. Indeed, we recall that at a discrete level the Wasserstein distance decreases along the heat flow: it is Inequality \eqref{equation_heat_flow_decreases_A}. Using the representation formula \eqref{equation_representation_A_sup} for the action, one concludes that for a fixed $s \geqslant 0$, 
\begin{equation}
\label{equation_heat_flow_decreases_A}
\int_0^1 \frac{1}{2} |\dot{\Phi_s \rho}_t |^2 \ddr t \leqslant \int_0^1 \frac{1}{2} |\dot{\rho}_t |^2 \ddr t.
\end{equation}   
However, by doing this, we lose the boundary values. To recover them, we squeeze the curve $\Phi_s \rho$ into the subinterval $\intervalleff{s}{1 - s}$, and then use the heat flow (acting on $\rho_0$) to join $\rho_0$ to $\Phi_s(\rho_0)$ on $\intervalleff{0}{s}$ and $\Phi_s(\rho_1)$ to $\rho_1$ on $\intervalleff{1-s}{1}$. Formally, for $0 < s \leqslant 1/2$, let us define the regularizing operator $R_s : \Gamma \to \Gamma$ by 
\begin{equation*}
\forall \rho \in \Gamma, \forall t \in \intervalleff{0}{1}, \ R_s(\rho)(t) := \begin{cases}
\Phi_t (\rho_0) & \text{if } 0 \leqslant t \leqslant s, \\
\Phi_s \left( \rho \left[ \frac{t-s}{1-2s} \right] \right) & \text{if } s \leqslant t \leqslant 1-s, \\
\Phi_{1 - t} (\rho_1) & \text{if } 1-s \leqslant t \leqslant 1.
\end{cases}
\end{equation*}    
The continuity of the heat flow allows us to assert that $R_s(\rho)$ is a continuous curve. As the entropy decreases along the heat flow (cf. \eqref{equation_EI}), and as $\Hen(R_s[\rho])$ is uniformly bounded on $\intervalleff{s}{1-s}$ (independently on $\rho$), we can see that there exists a constant $C_s$ depending only on $s$ such that
\begin{equation}
\label{equation_regularization_Linfty_entropy}
\forall \rho \in \Gamma, \ \forall t \in \intervalleff{0}{1}, \  \Hen[R_s(\rho)(t)] \leqslant \max(\Hen(\rho_0), \Hen(\rho_1), C_s).
\end{equation}  
To estimate the action of $R_s(\rho)$, we use the estimate \eqref{equation_heat_flow_decreases_A} on $\intervalleff{s}{1-s}$ and the identity \eqref{equation_EI} to hold the boundary terms: 
\begin{align*}
A(R_s(\rho)) & \leqslant \int_0^s \frac{1}{2} | \dot{\Phi_t \rho_0}|^2 \ddr t + \int_s^{1-s} \frac{1}{2} |\dot{\rho_{(t-s)/(1-2s)}} |^2 \ddr t + \int_{1-s}^{1} \frac{1}{2} | \dot{\Phi_{1-t} \rho_1}|^2 \ddr t \\
& = \frac{\Hen(\rho_0) - \Hen(\Phi_s [\rho_0])}{2} + \frac{1}{1-2s} \int_0^1 \frac{1}{2} |\dot{\rho_{t}} |^2 \ddr t + \frac{\Hen(\rho_1) - \Hen(\Phi_s [\rho_1])}{2} \\
& = \frac{1}{1-2s} A(\rho) +  \frac{1}{2}\left(\Hen(\rho_0) - \Hen(\Phi_s [\rho_0]) + \Hen(\rho_1) - \Hen(\Phi_s [\rho_1]) \right). 
\end{align*}
In particular, using the lower semi-continuity of the entropy $\Hen$ and the continuity w.r.t. $s$ of the heat flow, we see that if $\Hen(\rho_0)$ and $\Hen(\rho_1)$ are finite, 
\begin{equation}
\label{equation_action_decreases_regularization}
\limsup_{s \to 0} A(R_s(\rho)) \leqslant A(\rho). 
\end{equation}
We are now ready to use the regularization operator on the $W_2$-traffic plan $Q$. For a fixed $0 < s \leqslant 1/2$, we define $Q_s := R_s \# Q$. As $R_s$ does not change the boundary points, we still have $(e_0, e_1) \# Q_s = \gamma$. Integrating \eqref{equation_regularization_Linfty_entropy} w.r.t. $Q$, we get that 
\begin{equation*}
\forall t \in \intervalleff{0}{1}, \ H_{Q_s}(t)  \leqslant  H_{Q_s}(0) + H_{Q_s}(1) + C_s =  H_\gamma(0) +  H_\gamma(1) + C_s,
\end{equation*} 
and we know that the r.h.s. is finite because of Assumption \ref{assumption_finite_entropy_IF}. Concerning the action, since $\Hen(\rho_0)$ and $\Hen(\rho_1)$ are finite for $Q$-a.e. $\rho \in \Gamma$, we can integrate \eqref{equation_action_decreases_regularization} w.r.t. $Q$ by using a reverse Fatou's lemma to get
\begin{equation*}
\limsup_{s \to 0} \A(Q_s) \leqslant \A(Q).
\end{equation*}
It remains to check the incompressibility. For a fixed $t \in \intervalleff{0}{1}$, we notice that $e_t \# Q_s$ is of the form $(\Phi_r \circ e_{t'}) \# Q$ for a some $r \geqslant 0$ and $t' \in \intervalleff{0}{1}$ (for example, $r = t$ and $t'=0$ if $t \in \intervalleff{0}{s}$, and $r = s$ and $t' = (t-s)/(1-2s)$ if $t \in \intervalleff{s}{1-s}$). Thus, by linearity of the heat flow, $m_t(Q_s) = \Phi_r (m_{t'}[Q])$. But $m_{t'}(Q) = \Leb$ for any $t'$ and the Lebesgue measure is preserved by the heat flow, hence $m_t(Q_s) = \Leb$.   

Therefore, the $Q'$ that we take is just $Q_s$ for $s > 0$ small enough.  
\end{proof}

It is then possible to show how one can build a discrete curve from a continuous one in such a way that the action and the total entropy do not increase too much. This is a standard procedure which would be valid for probability on curves valued in arbitrary geodesic spaces.  

\begin{prop}
\label{proposition_gamma_limsup}
Let $Q \in \Proba^H(\Gamma)$ be an admissible $W_2$-traffic plan with finite total entropy. For any $N \geqslant 1$, we can build a $W_2$-traffic plan $Q_N \in \Probin(\Gamma_{T^N})$ in such a way that
\begin{equation*}
\limsup_{N \to + \infty} \A^{N, q, \lambda}(Q_N) \leqslant \A(Q) + \lambda \int_0^1 H_Q(t) \ddr t.
\end{equation*}
\end{prop}

\begin{proof}
We can assume that $\A(Q) < + \infty$. The idea is to sample each curve on a uniform grid, but not necessarily on $T^N$. Indeed, the key point in this sampling is to ensure that the discrete entropic penalization of the functional $\A^{N, q, \lambda}$ is bounded by $\lambda \int_0^1 H_Q(t) \ddr t$. Let us fix $N \geqslant 1$ and recall that $\tau = 1 / N$. We can see that
\begin{equation*}
\int_0^{\tau} \sum_{k=1}^{N-1} H_Q\left( k \tau + s \right) \ddr s  = \int_\tau^1 H_Q(t) \ddr t \leqslant \int_0^1 H_Q(t) \ddr t.
\end{equation*}
Therefore, there exists $s_N \in \intervalleoo{0}{\tau}$ such that
\begin{equation*}
\tau \sum_{k=1}^{N-1}  H_Q\left( k \tau + s_N \right)  \leqslant \int_0^1 H_Q(t) \ddr t. 
\end{equation*}
We define the sampling operator $S_N : \Gamma \to \Gamma_{T^N}$ (which samples on the grid $\dst{\left\{ k \tau + s_N \ :  \ k = 1, 2, \ldots, N-1 \right\}}$) by
\begin{equation*}
\forall \rho \in \Gamma, \forall k \in \{ 0, 1, \ldots, N \}, \  S_N(\rho) \left( k \tau \right) = \begin{cases}
\rho_0 & \text{if } k = 0, \\
\rho_1 & \text{if } k = N, \\
\rho_{k \tau + s_N} & \text{if } 1 \leqslant k \leqslant N-1.
\end{cases}
\end{equation*}
Then we simply define $Q_N := S_N \# Q$. As the initial and final values are left unchanged, it is clear that $(e_0, e_1) \# Q_N = (e_0, e_1) \# Q = \gamma$, i.e. $Q_N \in \Probbc(\Gamma_{T^N})$. By construction, we have that
\begin{equation*}
\lambda \sum_{k=1}^{N-1} \tau H_{Q_N}(k \tau)  = \lambda \tau \sum_{k=1}^{N-1}  H_Q\left( k \tau + s_N \right) \leqslant \lambda \int_0^1 H_Q(t) \ddr t. 
\end{equation*}
Moreover, as $Q \in \Proba(\Gamma)$ is incompressible, it is clear that $Q_N$ is incompressible too, hence 
\begin{equation*}
\sum_{k=1}^{N-1} \C_q(m_{k \tau}[Q_N]) = 0.   
\end{equation*} 
The last term to handle is the action. Indeed, we have to take care of the fact that we use a translated grid which is not uniform close to the boundaries. After a standard computation (which would be valid in any geodesic space) that we do not detail here, one finds that
\begin{equation*}
\sum_{k=1}^N \int_{\Gamma_{T^N}} \frac{W_2^2(\rho_{(k-1) \tau}, \rho_{k \tau})}{2 \tau} \ddr Q_N(\rho) \leqslant \A(Q) + \int_\Gamma \left( \int_{0}^{2 \tau} \frac{1}{2} |\dot{\rho}_s|^2 \ddr s \right) \ddr Q (\rho).
\end{equation*} 
For every $2$-absolutely continuous curve, it is clear that the quantity $\int_{0}^{2 \tau} \frac{1}{2} |\dot{\rho}_s|^2 \ddr s$ goes to $0$ as $N \to + \infty$ and it is dominated by $A(\rho)$ which is integrable w.r.t. $Q$. Therefore, by dominated convergence, 
\begin{equation*}
\limsup_{N \to + \infty} \left( \sum_{k=1}^N \int_{\Gamma_{T^N}} \frac{W_2^2(\rho_{(k-1) \tau}, \rho_{k \tau})}{2 \tau} \ddr Q_N(\rho) \right) \leqslant \A(Q).  
\end{equation*}  
Gluing all the inequalities we have collected on $Q_N$, we see that $\A^{N, q, \lambda}(Q_N)$ satisfies the desired asymptotic bound. 
\end{proof}

\begin{crl}
\label{corollary_uniform_bound_discrete}
Under Assumption \ref{assumption_L1_entropy} or Assumption \ref{assumption_finite_entropy_IF}, there exists $C < + \infty$, such that, uniformly in $N \geqslant 1$, $\lambda \in \intervalleof{0}{1}$ and $q>1$, we have 
\begin{equation*}
\A^{N, q, \lambda}(\bar{Q}^{N, q, \lambda}) \leqslant C. 
\end{equation*} 
\end{crl}

\begin{proof}
Indeed, it is enough to take $Q$ any element of $\Proba^H(\Gamma)$ with finite action (it exists by definition under Assumption \ref{assumption_L1_entropy} and we use Proposition \ref{proposition_approximation_finite_entropy} under Assumption \ref{assumption_finite_entropy_IF}), to construct $Q_N$ as in Proposition \ref{proposition_gamma_limsup}, to define $C := \sup_{N \geqslant 1} \A^{N, q, \lambda}(Q_N)$, and to use the fact that $\A^{N, q, \lambda}(\bar{Q}^{N, q, \lambda}) \leqslant \A^{N, q, \lambda}(Q_N) \leqslant C$. 
\end{proof}

\subsection{Solution of the continuous problem as a limit of discrete solutions}

To go from $W_2$-traffic plans on discrete curves to $W_2$-traffic plans on continuous ones, we will need an extension operator $E_N : \Gamma_{T^N} \to \Gamma$ that interpolates a discrete curve along geodesics in $(\Prob(\Omega), W_2)$. More precisely, 

\begin{defi}
Let $N \geqslant 1$. If $\rho \in \Gamma_{T^N}$, the curve $E_N(\rho) \in \Gamma$ is defined as the one that coincides with $\rho$ on $T^N$ and such that for any $k \in \{ 0, 1, \ldots, N-1 \}$, the restriction of $E_N(\rho)$ to $\intervalleff{k \tau}{(k+1) \tau}$ is a\footnotemark{} constant-speed geodesic joining $\rho_{k \tau}$ to $\rho_{(k+1) \tau}$. 
\end{defi}

\footnotetext{One may worry about the non uniqueness of the geodesic and hence of the fact that the extension operator $E_N$ is ill-defined. However, it is a classical result of optimal transport that the constant-speed geodesic joining two measures is unique as soon as one of the two measures is absolutely continuous w.r.t. $\Leb$. Moreover, for a traffic plan $Q \in \Prob(\Gamma_{T^N})$, if $H_Q(t) < + \infty$ for $t \in T^N$, then $Q$-a.e. $\rho$ is absolutely continuous w.r.t. $\Leb$ at time $t$. Thus as long as we work with $W_2$-traffic pans $Q$ such that $H_Q(k \tau) < + \infty$ for any $k \in \{ 1,2, \ldots, N-1 \}$ (and we leave it to the reader to check that it is the case), the operator $E_N$ is well defined.}

\noindent In particular, for any $k \in \{ 0, 1, 2, \ldots, N-1 \}$, $|\dot{E_N(\rho)}|$ is constant on $\intervalleff{k \tau}{(k+1) \tau}$ and equal to $W_2(\rho_{k \tau}, \rho_{(k+1) \tau})/ \tau$. Thus, we have the identity 
\begin{equation*}
\int_{k \tau}^{(k+1) \tau} \frac{1}{2} |\dot{E_N (\rho)}_t|^2 \ddr t = \frac{W_2^2(\rho_{k \tau}, \rho_{(k+1) \tau})}{2 \tau},
\end{equation*}
summed over $k \in \{ 0, 1, \ldots, N-1 \}$, these identities led to 
\begin{equation}
\label{equation_identity_A_discrete_continuous}
A(E_N[\rho]) = \sum_{k=1}^N \frac{W_2^2(\rho_{(k-1) \tau}, \rho_{k \tau})}{2 \tau}.
\end{equation} 
In other words, the action of the extended curve $E_N(\rho)$ is equal to the discrete one of $\rho$.   

\bigskip

We are now ready to show the convergence of $\bar{Q}^{N, q, \lambda}$ to some limit $\bar{Q} \in \Proba(\Gamma)$. We take three sequences $(N_n)_{n \in \N}$, $(q_m)_{m \in \N}$ and $(\lambda_r)_{r \in \N}$ that converge respectively to $+ \infty$, $+ \infty$ and $0$. We will not relabel the sequences when extracting subsequences. Moreover, to avoid heavy notations, we will drop the indexes $n, m$ and $r$, and $\lim_{n \to + \infty}$, $\lim_{m \to + \infty}$, $\lim_{r \to + \infty}$ will be denoted respectively by $\lim_{N \to + \infty}$, $\lim_{q \to + \infty}$ and $\lim_{\lambda \to 0}$.  

\begin{prop}
\label{proposition_convergence_barQ}
Under Assumption \ref{assumption_L1_entropy} or Assumption \ref{assumption_finite_entropy_IF}, there exists $\bar{Q} \in \Proba(\Gamma)$, and families $(\bar{Q}^{N, \lambda})_{N, \lambda} \in \Proba(\Gamma_{T^N})$, $(\bar{Q}^\lambda)_\lambda \in \Proba(\Gamma)$ such that (up to extraction) 
\begin{alignat*}{5}
\lim_{q \to + \infty} \bar{Q}^{N, q, \lambda} &=  \bar{Q}^{N, \lambda}  && \text{     in }  &\Prob(\Gamma_{T^N}), \\
\lim_{N \to + \infty} (E_N \# \bar{Q}^{N, \lambda}) &=    \bar{Q}^{\lambda} && \text{     in } &\Prob(\Gamma), \\
\lim_{\lambda \to 0} \bar{Q}^{\lambda} &=  \bar{Q} &&  \text{     in } &\Prob(\Gamma).
\end{alignat*}
\end{prop} 

\begin{proof}
We denote by $C$ the constant given by Corollary \ref{corollary_uniform_bound_discrete}.  

To prove the existence of $(\bar{Q}^{N, \lambda})_{N, \lambda}$, it is enough to notice that for any $N \geqslant 1$ the space $\Prob(\Gamma_{T^N})$ is compact and therefore every sequence admits a converging subsequence. By continuity of the Wasserstein distance, we know that  
\begin{equation*}
\sum_{k=1}^N \int_{\Gamma_{T^N}} \frac{W_2^2(\rho_{(k-1) \tau}, \rho_{k \tau})}{2 \tau} \ddr \bar{Q}^{N, \lambda}(\rho) = \lim_{q \to + \infty} \left( \sum_{k=1}^N \int_{\Gamma_{T^N}} \frac{W_2^2(\rho_{(k-1) \tau}, \rho_{k \tau})}{2 \tau} \ddr \bar{Q}^{N, q, \lambda}(\rho) \right) 
\leqslant C. 
\end{equation*}
To go on, we use \eqref{equation_identity_A_discrete_continuous}, namely the fact that $E_N$ transforms the discrete action into the continuous one:
\begin{equation*}
\A(E_N \# \bar{Q}^{N, \lambda})  = \int_{\Gamma_{T^N}} A(E_N(\rho)) \ddr \bar{Q}^{N, \lambda}(\rho) 
 =  \sum_{k=1}^N \int_{\Gamma_{T^N}} \frac{W_2^2(\rho_{(k-1) \tau}, \rho_{k \tau})}{2 \tau} \ddr \bar{Q}^{N, \lambda}(\rho) 
\leqslant C.
\end{equation*}
We know that the functional $\A$ is l.s.c. and that its sublevel sets are compact. Hence, we get the existence of $(\bar{Q}^\lambda)_\lambda$ such that
\begin{equation*}
\lim_{N \to + \infty} (E_N \# \bar{Q}^{N, \lambda}) = \bar{Q}^{\lambda}
\end{equation*} 
in $\Prob(\Gamma)$ and $\A(\bar{Q}^\lambda) \leqslant C$. Applying exactly the same argument, we can conclude at the existence of $\bar{Q} \in \Prob(\Gamma)$ with 
\begin{equation*}
\lim_{\lambda \to 0} \bar{Q}^{\lambda} = \bar{Q}
\end{equation*}
in $\Prob(\Gamma)$ together with $\A(\bar{Q}) \leqslant C$. 

It is easy to show that $(e_0, e_1) \# \bar{Q} = \gamma$ as we have that $(e_0, e_1) \# \bar{Q}^{N, q, \lambda} = \gamma$: this condition passes to the limit and is preserved by $E_N$. 

It is slightly more difficult to show the incompressibility. Let us first show that $\bar{Q}^{N, \lambda} \in \Probin(\Gamma_{T^N})$. We fix $N, k$ and $\lambda$. As $\A^{N, q, \lambda}(\bar{Q}^{N, q, \lambda})\leqslant C$, we see that  
\begin{equation*}
\left( \int_\Omega |m_{k \tau}( \bar{Q}^{N, q, \lambda} )|^q \right)^{1/q} \leqslant (C+1)^{1/q}.
\end{equation*}
As $\Leb(\Omega) = 1$, the $L^q$ norms are increasing with $q$. Thus, for any $q_0 \leqslant q$, we have 
\begin{equation*}
\left( \int_\Omega |m_{k \tau}( \bar{Q}^{N, q, \lambda} )|^{q_0} \right)^{1/q_0} \leqslant (C+1)^{1/q}.
\end{equation*} 
Let us take the limit $q \to + \infty$. We have that $m_{k \tau}( \bar{Q}^{N, q, \lambda} )$ converges in $\Prob(\Omega)$ to $m_{k \tau}(\bar{Q}^{N, \lambda})$. As the $L^{q_0}$ norm is l.s.c. w.r.t. the weak convergence of measures, we can see that 
\begin{equation*}
\left( \int_\Omega |m_{k \tau}( \bar{Q}^{N,\lambda} )|^{q_0} \right)^{1/q_0} \leqslant 1.
\end{equation*}
But now $q_0$ is arbitrary, thus the $L^\infty$ norm of $m_{k \tau}(\bar{Q}^{N, \lambda})$ is bounded by $1$. As we know that $m_{k \tau}(\bar{Q}^{N, \lambda})$ is a probability measure and that $\Leb(\Omega) = 1$, we deduce that $m_{k \tau}(\bar{Q}^{N, \lambda})$ is equal to $1$ $\Leb$-a.e. on $\Omega$: it exactly means that $m_{k \tau}(\bar{Q}^{N, \lambda}) = \Leb$. As $\bar{Q}^{N, \lambda}$ also satisfies the boundary conditions, $\bar{Q}^{N, \lambda} \in \Proba(\Gamma_{T^N})$.  

To show that the incompressibility constraint is satisfied by $\bar{Q}^{\lambda}$ for every $t$, we proceed as follows: let us consider $t \in \intervalleff{0}{1}$ and $N \geqslant 1$. Let $k \in \{ 0, 1, \ldots, N-1 \}$ such that $\dst{k \tau \leqslant t \leqslant (k+1) \tau }$. We denote by $s \in \intervalleff{0}{1}$ the real such that $\dst{t = (k+s) \tau}$. By definition of $E_N$, if $\rho \in \Gamma_{T^N}$, there exists $\bar{\gamma}$ an optimal transport plan between $\rho_{k \tau}$ and $\rho_{(k+1) \tau}$ (i.e. an optimal $\gamma$ in formula \eqref{equation_definition_Wasserstein_distance} with $\mu = \rho_{k \tau}$ and $\nu = \rho_{(k+1) \tau}$) such that $E_N(\rho)(t) = \pi_s \# \bar{\gamma}$ with $\pi_s : (x,y) \mapsto (1-s)x + sy$. For any $a \in C^1(\Omega)$, we can see that   
\begin{align*}
\left| \int_\Omega a \ddr [E_N(\rho)(t)] - \int_\Omega a \ddr \rho_{k \tau} \right| & = \left| \int_{\Omega \times \Omega} ( a[(1-s)x + s y] - a[x] ) \ddr \bar{\gamma}(x,y) \right| \\
& \leqslant \int_{\Omega \times \Omega} s |\nabla a(x)| |x-y| \ddr \bar{\gamma}(x,y) \\
& \leqslant \sqrt{\int_{\Omega \times \Omega} |\nabla a(x)|^2 \ddr \bar{\gamma}(x,y)} \sqrt{\int_{\Omega \times \Omega} |x-y|^2 \ddr \bar{\gamma}(x,y)} \\
& \leqslant \| \nabla a \|_{L^\infty} W_2( \rho_{k \tau}, \rho_{(k+1) \tau}).
\end{align*}   
Therefore, if we estimate the action of $m_t(E_N \# \bar{Q}^{N, \lambda})$ on a $C^1$ function $a$, we find that 
\begin{align*}
\left| \int_\Omega a \ddr[m_t(E_N \# \bar{Q}^{N, \lambda})] - \int_\Omega a(x) \ddr x \right| &= \left| \int_\Omega a \ddr[m_t(E_N \# \bar{Q}^{N, \lambda})] - \int_\Omega a \ddr[m_{k \tau}(\bar{Q}^{N, \lambda})] \right| \\
& \leqslant \int_{\Gamma_{T^N}} \left| \int_\Omega a \ddr [ E_N(\rho)(t)] - \int_\Omega a \ddr \rho_{k \tau} \right| \ddr \bar{Q}^{N, \lambda}(\rho) \\
& \leqslant \| \nabla a \|_{L^\infty} \int_{\Gamma_{T^N}} W_2(\rho_{k \tau}, \rho_{(k+1) \tau})  \ddr \bar{Q}^{N, \lambda}(\rho) \\
& \leqslant \sqrt{2\tau} \| \nabla a \|_{L^\infty} \sqrt{\int_{\Gamma_{T^N}} \frac{W_2^2(\rho_{k \tau}, \rho_{(k+1) \tau})}{2 \tau}  \ddr \bar{Q}^{N, \lambda}(\rho)}  \\
& \leqslant \sqrt{2 C \tau} \| \nabla a \|_{L^\infty}.
\end{align*}
Taking the limit $N \to + \infty$ (hence $\tau \to 0$), we know that $m_t(E_N \# \bar{Q}^{N, \lambda})$ converges to $m_t(\bar{Q}^\lambda)$, thus we get 
\begin{equation*}
\int_\Omega a \ddr [m_t(\bar{Q}^\lambda)] = \int_\Omega a(x) \ddr x.
\end{equation*}
As $a$ is an arbitrary $C^1$ function, we have the equality $m_t(\bar{Q}^\lambda) = \Leb$ for any $t$, in other words, $\bar{Q}^\lambda \in \Probin(\Gamma)$. As we already know that $\bar{Q}^\lambda \in \Probbc(\Gamma)$, we conclude that $\bar{Q}^\lambda \in \Proba(\Gamma)$ for any $\lambda > 0$. But $\Proba(\Gamma)$ is closed, therefore $\bar{Q} \in \Proba(\Gamma)$.  
\end{proof}

With all the previous work, it is easy to conclude that $\bar{Q}$ is a minimizer of $\A$: we just copy a standard proof of $\Gamma$-convergence. 

\begin{prop}
\label{proposition_optimality_barQ}
Under Assumption \ref{assumption_L1_entropy} or Assumption \ref{assumption_finite_entropy_IF}, $\bar{Q}$ is a solution of the continuous problem \eqref{equation_continuous_problem}. 
\end{prop}

\begin{proof}
We have already seen that $\dst{\A(E_N \# \bar{Q}^{N, \lambda}) \leqslant \liminf_{q \to + \infty} \A^{N,q,\lambda}(\bar{Q}^{N, q, \lambda})}$. By lower semi-continuity of $\A$, we deduce that 
\begin{equation*}
\A(\bar{Q}) \leqslant \liminf_{\lambda \to 0} \left( \liminf_{N \to + \infty} \left( \liminf_{q \to + \infty} \A^{N, q, \lambda}(\bar{Q}^{N, q, \lambda}) \right) \right).
\end{equation*} 
By contradiction, let us assume that there exists $Q \in \Proba(\Gamma)$ such that $\A(Q) < \A(\bar{Q})$. If we are under Assumption \ref{assumption_finite_entropy_IF}, we can regularize it thanks to Proposition \ref{proposition_approximation_finite_entropy}, and under Assumption \ref{assumption_L1_entropy} we know that we can assume that $Q' \in \Proba^H(\Gamma)$ and $\A(Q') \leqslant \A(Q)$. In any of these two cases, we can assume that there exists $Q \in \Proba^H(\Gamma)$ such that $\A(Q) < \A(\bar{Q})$. Thanks to Proposition \ref{proposition_gamma_limsup}, we know that we can construct a sequence $Q_N$ with 
\begin{equation*}
\limsup_{N \to + \infty} \A^{N, q, \lambda}(Q_N) \leqslant \A(Q) + \lambda \int_0^1 H_Q(t) \ddr t
\end{equation*}
Taking the limit $\lambda \to 0$ and using $\A(Q) < \A(\bar{Q})$, we get 
\begin{equation*}
\limsup_{\lambda \to 0} \left( \limsup_{N \to + \infty} \A^{N, q, \lambda}(Q_N) \right) < \liminf_{\lambda \to 0} \left( \liminf_{N \to + \infty} \left( \liminf_{q \to + \infty} \A^{N, q, \lambda}(\bar{Q}^{N, q, \lambda}) \right) \right).
\end{equation*}
Taking $N$ and $q$ large enough and $\lambda$ small enough, one has $\A^{N, q, \lambda}(Q_N) < \A^{N, q, \lambda}(\bar{Q}^{N, q, \lambda})$, which contradicts the optimality of $\bar{Q}^{N, q, \lambda}$. 
\end{proof}

\subsection{Behavior of the averaged entropy of \texorpdfstring{$\bar{Q}$}{Qbar}}

Now, we will show that $H_{\bar{Q}} \in L^1(\intervalleff{0}{1})$ and that $\bar{Q}$ is the minimizer of $\A$ with minimal total entropy. If $Q \in \Prob(\Gamma_{T^N})$, let us denote by $H^\text{int}_Q : \intervalleff{0}{1} \to \intervalleff{0}{+ \infty}$ the piecewise affine interpolation of $H_Q$. More precisely, if $k \in \{ 0, 1, \ldots, N-1 \}$ and $s \in \intervalleff{0}{1}$, we define 
\begin{equation*}
H^\text{int}_Q \left( (k+s) \tau \right) := (1-s) H_Q \left( k \tau \right) + s H_Q \left( (k+1) \tau \right). 
\end{equation*}
We show the following estimate, which relies on the lower semi-continuity of the entropy: 

\begin{prop}
\label{proposition_upper_bound_continuous_entropy}
For any $t \in \intervalleff{0}{1}$, we have the following upper bound for $H_{\bar{Q}}(t)$: 
\begin{equation*}
H_{\bar{Q}}(t) \leqslant \liminf_{\lambda \to 0} \left( \liminf_{N \to + \infty} \left( \liminf_{q \to + \infty} H^\text{int}_{\bar{Q}^{N, q, \lambda}}(t) \right) \right). 
\end{equation*}
\end{prop}

\begin{proof}
For a fixed $t$, $Q \mapsto H_Q(t)$ is l.s.c. (Lemma \ref{lemma_lsc_compact_averaged}). Thus, for any $k \in \{ 0, 1, 2, \ldots, N \}$, we have
\begin{equation*}
H_{\bar{Q}^{N, \lambda}} \left( k \tau \right) \leqslant \liminf_{q \to + \infty} H_{\bar{Q}^{N, q, \lambda}} \left( k \tau \right). 
\end{equation*}
Then, to pass to the limit $N \to + \infty$, we will use the fact that the entropy is geodesically convex, i.e. convex along the constant-speed geodesics. Recall that $E_N : \Gamma_{T^N} \to \Gamma$ is the extension operator that interpolates along constant-speed geodesics. Let us take $\rho \in \Gamma_{T^N}$. By geodesic convexity, we have for any $k \in \{ 0, 1, \ldots, N-1 \}$ and $s \in \intervalleff{0}{1}$
\begin{equation*}
\Hen \left[ E_N(\rho)\left( (k+s) \tau \right) \right] \leqslant (1-s) \Hen(\rho_{k \tau}) + s \Hen(\rho_{(k+1) \tau}).
\end{equation*} 
Integrating this inequality over $\Gamma_{T^N}$ w.r.t. $\bar{Q}^{N, \lambda}$, we get 
\begin{align*}
H_{E_N \# \bar{Q}^{N, \lambda}} \left( (k+s) \tau \right) & \leqslant (1-s) H_{\bar{Q}^{N, \lambda}}\left( k \tau \right) + s H_{\bar{Q}^{N, \lambda}}\left( (k+1) \tau \right) \\
& \leqslant \liminf_{q \to + \infty} \left[ (1-s) H_{\bar{Q}^{N, q, \lambda}}\left( k \tau \right) + s H_{\bar{Q}^{N,q, \lambda}}\left( (k+1) \tau \right) \right] \\
& = \liminf_{q \to + \infty} \left[ H^\text{int}_{\bar{Q}^{N,q, \lambda}} \left( (k+s) \tau \right) \right]   
\end{align*}
We take the limit $N \to + \infty$, followed by $\lambda \to 0$ to get (thanks to the lower semi-continuity of the averaged entropy) the announced inequality. 
\end{proof}

We derive a useful consequence, which implies Theorem \ref{theorem_finite_entropy_IF_implies_L1}.

\begin{crl}
\label{corollary_finite_entropy_IF_implies_bounded}
Under Assumption \ref{assumption_finite_entropy_IF}, the function $H_{\bar{Q}}$ is bounded by $\max(H_\gamma(0), H_\gamma(1))$. 
\end{crl}  

\begin{proof}
This is where we use the work of Section \ref{section_discrete problem}: thanks to Theorem \ref{theorem_convexity_entropy_discrete}, we know that $H_{\bar{Q}^{N, q, \lambda}}$ is convex and therefore bounded by the values at its endpoints which happen to be finite (independently of $N, q$ or $\lambda$): 
\begin{equation*}
\forall k \in \{ 0, 1, 2, \ldots, N \}, \ H_{\bar{Q}^{N, q, \lambda}} \left( k \tau \right) \leqslant \max(H_\gamma(0), H_\gamma(1)).  
\end{equation*} 
Thus the function $H^\text{int}_{\bar{Q}^{N, q, \lambda}}$ is also bounded uniformly on $\intervalleff{0}{1}$ by $\max(H_\gamma(0), H_\gamma(1))$. Proposition \ref{proposition_upper_bound_continuous_entropy} allows us to conclude that the same bound holds for $H_{\bar{Q}}$. 
\end{proof}

As we have now proved Theorem \ref{theorem_finite_entropy_IF_implies_L1}, we will work only under Assumption \ref{assumption_L1_entropy}. It remains to show that the $\bar{Q}$ we constructed is the one with minimal total entropy. This is done thanks to the entropic penalization, and is standard in $\Gamma$-convergence theory, the specific structure of the Wasserstein space does not play any role. 

\begin{prop}
\label{proposition_barQ_lowest_entropy}
For any $Q \in \Proba^H(\Gamma)$ solution of the continuous problem \eqref{equation_continuous_problem}, we have 
\begin{equation*}
\int_0^1 H_{\bar{Q}}(t) \ddr t \leqslant \int_0^1 H_Q(t) \ddr t
\end{equation*}
\end{prop}

\begin{proof}
Let us start with an exact quadrature formula for $H^\text{int}_{\bar{Q}^{N, q, \lambda}}$ :   
\begin{equation*}
\int_{\tau}^{1 - \tau} H^\text{int}_{\bar{Q}^{N, q, \lambda}}(t) \ddr t 
=  \frac{\tau}{2} H_{\bar{Q}^{N, q, \lambda}} \left( \tau \right) + \tau \sum_{k=2}^{N-2} H_{\bar{Q}^{N, q, \lambda}} \left( k \tau \right) + \frac{\tau}{2} H_{\bar{Q}^{N, q, \lambda}} \left( 1 - \tau \right)
\leqslant \tau \sum_{k=1}^{N-1} H_{\bar{Q}^{N, q, \lambda}} \left( k \tau \right) 
\end{equation*}
Then we take successively the limits $q \to + \infty$, $N \to + \infty$ and $\lambda \to 0$, applying Fatou's lemma and using Proposition \ref{proposition_upper_bound_continuous_entropy} to get 
\begin{equation}
\label{equation_bound_HQbar_liminf3}
\int_0^1 H_{\bar{Q}}(t) \ddr t \leqslant \liminf_{\lambda \to 0} \left( \liminf_{N \to + \infty} \left( \liminf_{q \to + \infty} \left( \tau \sum_{k=1}^{N-1} H_{\bar{Q}^{N, q, \lambda}} \left[ k \tau \right] \right) \right) \right). 
\end{equation}
On the other hand, let us show that the r.h.s. of \eqref{equation_bound_HQbar_liminf3} is smaller than the total entropy of any minimizer of \eqref{equation_continuous_problem}. Indeed, assume that this is not the case for some $Q \in \Proba(\Gamma)$ solution of \eqref{equation_continuous_problem}. In particular, for some $\lambda > 0$ small enough, we have the strict inequality 
\begin{equation*}
\int_0^1 H_Q(t) \ddr t <  \liminf_{N \to + \infty} \left( \liminf_{q \to + \infty} \left( \tau \sum_{k=1}^{N-1} H_{\bar{Q}^{N, q, \lambda}} \left[ k \tau \right] \right) \right). 
\end{equation*}
Using the fact that $\A(Q) \leqslant \A(\bar{Q}^\lambda)$ by optimality of $Q$, and thanks to the lower semi-continuity of the action, 
\begin{equation*}
\A(Q) \leqslant \A(\bar{Q}^\lambda) \leqslant \liminf_{N \to + \infty} \left( \liminf_{q \to + \infty} \left( \sum_{k=1}^N \int_{\Gamma_{T^N}} \frac{W_2^2(\rho_{(k-1) \tau}, \rho_{k \tau})}{2 \tau} \ddr \bar{Q}^{N, q, \lambda}(\rho) \right) \right). 
\end{equation*}
Therefore, gluing these two estimates together, we obtain 
\begin{equation*}
\A(Q) + \lambda \int_0^1 H_Q(t) \ddr t < \liminf_{N \to + \infty} \left( \liminf_{q \to + \infty} \left( \sum_{k=1}^N \int_{\Gamma_{T^N}} \frac{W_2^2(\rho_{(k-1) \tau}, \rho_{k \tau})}{2 \tau} \ddr \bar{Q}^{N, q, \lambda}(\rho) +   \lambda \sum_{k=1}^{N-1} \tau H_{\bar{Q}^{N, q, \lambda}} [ k \tau ] \right) \right). 
\end{equation*}
But if we build the $Q_N$ from $Q$ as in Proposition \ref{proposition_gamma_limsup}, we get, for $N$ and $q$ large enough, 
\begin{equation*}
\A^{N, q, \lambda}(Q_N) <  \sum_{k=1}^N \int_{\Gamma_{T^N}} \frac{W_2^2(\rho_{(k-1) \tau}, \rho_{k \tau})}{2 \tau} \ddr \bar{Q}^{N, q, \lambda}(\rho) +   \lambda \sum_{k=1}^{N-1} \tau H_{\bar{Q}^{N, q, \lambda}} ( k \tau ) \leqslant \A^{N, q, \lambda}(\bar{Q}^{N, q, \lambda}),
\end{equation*} 
which is a contradiction with the optimality of $\bar{Q}^{N, q, \lambda}$. Hence, we have proved that for any $Q \in \Proba(\Gamma)$ solution of the continuous problem, 
\begin{equation}
\label{equation_encadrement_liminf3}
\int_0^1 H_{\bar{Q}}(t) \ddr t \leqslant \liminf_{\lambda \to 0} \left( \liminf_{N \to + \infty} \left( \liminf_{q \to + \infty} \left( \tau \sum_{k=1}^{N-1} H_{\bar{Q}^{N, q, \lambda}} \left[ k \tau \right] \right) \right) \right) \leqslant \int_0^1 H_{Q}(t) \ddr t. \qedhere
\end{equation}
\end{proof}

Now it remains to show that $H_{\bar{Q}}$ is a convex function of time. This will be done by proving that $H_{\bar{Q}}$ is the limit of $H^\text{int}_{\bar{Q}^{N, q, \lambda}}$.

\begin{prop}
\label{proposition_convexity_ae_entropy}
Under Assumption \ref{assumption_L1_entropy}, for a.e. $t \in \intervalleff{0}{1}$,
\begin{equation*}
H_{\bar{Q}}(t) = \lim_{\lambda \to 0} \left( \lim_{N \to + \infty} \left( \lim_{q \to + \infty} \left( H^\text{int}_{\bar{Q}^{N, q, \lambda}}(t) \right) \right) \right).
\end{equation*}
\end{prop}

\begin{proof}
Taking $Q = \bar{Q}$ in \eqref{equation_encadrement_liminf3}, we see that, up to extraction, 
\begin{equation*}
\int_0^1 H_{\bar{Q}}(t) \ddr t = \lim_{\lambda \to 0} \left( \lim_{N \to + \infty} \left( \lim_{q \to + \infty} \left( \tau \sum_{k=1}^{N-1} H_{\bar{Q}^{N, q, \lambda}} \left[ k \tau \right] \right) \right) \right). 
\end{equation*}
In other words, the integral over time of the discrete averaged entropy converges to the integral of the continuous one. As we know moreover that the discrete averaged entropy is an upper bound for the continuous one (Proposition \ref{proposition_upper_bound_continuous_entropy}), it is not difficult to show that the discrete averaged entropy converges (up to extraction) a.e. to the continuous one.  
\end{proof}

\subsection{From convexity a.e. to true convexity}

Proposition \ref{proposition_convexity_ae_entropy} is slightly weaker than the result we claimed, as we get information about $H_{\bar{Q}}$ only for a.e. time. The first step toward true convexity is to show that, under Assumption \ref{assumption_finite_entropy_IF}, the averaged entropy is everywhere below the line joining the endpoints. 

\begin{prop}
\label{proposition_true_convexity_bdp}
Under Assumption \ref{assumption_finite_entropy_IF}, for any $t \in \intervalleff{0}{1}$, we have 
\begin{equation*}
H_{\bar{Q}}(t) \leqslant (1-t) H_{\bar{Q}}(0) + t H_{\bar{Q}}(1). 
\end{equation*}
\end{prop} 

\begin{proof}
From Proposition \ref{proposition_convexity_ae_entropy}, we know that $H_{\bar{Q}}$ is a.e. the limit of the functions $H^\text{int}_{\bar{Q}^{N,q, \lambda}}$. Thanks to Theorem \ref{theorem_convexity_entropy_discrete}, we can assert that for any $t \in \intervalleff{0}{1}$, one has $H^\text{int}_{\bar{Q}^{N,q, \lambda}}(t) \leqslant (1-t) H^\text{int}_{\bar{Q}^{N,q, \lambda}}(0) + t H^\text{int}_{\bar{Q}^{N,q, \lambda}}(1) $. We also know that $H_{\bar{Q}}$ and $H^\text{int}_{\bar{Q}^{N,q, \lambda}}$ coincide for $t = 0$ and $t=1$. Therefore, for a.e. $t \in \intervalleff{0}{1}$, 
\begin{align*}
H_{\bar{Q}}(t) & =  \lim_{\lambda \to 0} \left( \lim_{N \to + \infty} \left( \lim_{q \to + \infty} \left( H^\text{int}_{\bar{Q}^{N, q, \lambda}}(t) \right) \right) \right) \\
& \leqslant  \lim_{\lambda \to 0} \left( \lim_{N \to + \infty} \left( \lim_{q \to + \infty} \left( (1-t) H^\text{int}_{\bar{Q}^{N,q, \lambda}}[0] + t H^\text{int}_{\bar{Q}^{N,q, \lambda}}[1]  \right) \right) \right) \\
& = (1-t) H_{\bar{Q}}(0) + t H_{\bar{Q}}(1).
\end{align*}
As $H_{\bar{Q}}$ is l.s.c., we see that the above inequality is valid for any $t \in \intervalleff{0}{1}$. 
\end{proof}

Now, if $\bar{Q}$ is the solution of the continuous problem \eqref{equation_continuous_problem} with minimal total entropy, then its restriction to any subinterval of $\intervalleff{0}{1}$ is also optimal: for any $0 \leqslant t_1 < t_2 \leqslant 1$, $e_{\intervalleff{t_1}{t_2}} \# \bar{Q}$ is also the solution of the continuous problem (on $\intervalleff{t_1}{t_2}$) with boundary conditions $e_{\{t_1, t_2 \}} \# \bar{Q}$ with minimal total entropy. This is already known \cite[Remark 3.2 and below]{Ambrosio2009} and comes from the fact that we can concatenate traffic plans. 

\begin{prop}
\label{proposition_gluing_solution}
Let $0 \leqslant t_1 < t_2 \leqslant 1$. Then for any $Q \in \Proba(\Gamma_{\intervalleff{t_1}{t_2}})$ such that  $e_{\{t_1, t_2 \}} \# Q = e_{\{t_1, t_2 \}} \# \bar{Q}$, we have 
\begin{equation*}
\int_\Gamma \left( \int_{t_1}^{t_2} \frac{1}{2} |\dot{\rho}_t|^2 \ddr t \right) \ddr \bar{Q}(\rho) \leqslant \int_{\Gamma_{\intervalleff{t_1}{t_2}}} \left( \int_{t_1}^{t_2} \frac{1}{2} |\dot{\rho}_t|^2 \ddr t \right) \ddr Q(\rho).
\end{equation*}
Moreover, if the inequality above is an equality, then 
\begin{equation*}
\int_{t_1}^{t_2} H_{\bar{Q}}(t) \ddr t \leqslant \int_{t_1}^{t_2} H_{Q}(t) \ddr t.
\end{equation*}
\end{prop}

\begin{proof}
This property relies on the fact that if $Q \in \Proba(\Gamma_{\intervalleff{t_1}{t_2}})$ with $e_{\{t_1, t_2 \}} \# Q = e_{\{t_1, t_2 \}} \# \bar{Q}$, we can concatenate $Q$ and $\bar{Q}$ together to build a $W_2$-traffic plan $Q' \in \Prob(\Gamma)$ such that $e_{\intervalleff{0}{1} \bsl \intervalleff{t_1}{t_2}} \# Q' = e_{\intervalleff{0}{1} \bsl \intervalleff{t_1}{t_2}} \# \bar{Q}$ and $e_{\intervalleff{t_1}{t_2}} \# Q' = e_{\intervalleff{t_1}{t_2}} \# Q$. To do that, it is enough to disintegrate the measures $\bar{Q}$ and $Q$ w.r.t. $e_{\{t_1, t_2 \}}$ and then to concatenate elements of $\Gamma_{\intervalleff{0}{1} \bsl \intervalleff{t_1}{t_2}}$ and $\Gamma_{\intervalleff{t_1}{t_2}}$ which coincides on $\{ t_1, t_2 \}$: we leave the details to the reader.  
\end{proof}

Combining the two above propositions, we recover the convexity of $H_{\bar{Q}}$. Let us remark that we rely on the fact that the minimizer of $\A$ with minimal total entropy is unique. 

\begin{crl}
\label{corollary_barQ_true_convex_entropy}
Under Assumption  \ref{assumption_L1_entropy} or Assumption \ref{assumption_finite_entropy_IF}, for any $0 \leqslant t_1 < t_2 \leqslant 1$ and any $s \in \intervalleoo{0}{1}$,
\begin{equation*}
H_{\bar{Q}}((1-s)t_1 + st_2) \leqslant (1-s) H_{\bar{Q}}(t_1) + s H_{\bar{Q}}(t_2). 
\end{equation*} 
\end{crl}

\begin{proof}
If the r.h.s. is infinite, there is nothing to prove. Therefore, we can assume that $H_{\bar{Q}}(t_1)$ and $H_{\bar{Q}}(t_2)$ are finite. By uniqueness of the solution with minimal total entropy (Proposition \ref{proposition_strict_convexity_entropy}), we know that $e_{\intervalleff{t_1}{t_2}} \#  \bar{Q}$ coincides with the solution of the continuous problem \eqref{equation_continuous_problem} with minimal total entropy on $\intervalleff{t_1}{t_2}$ with boundary conditions $e_{\{ t_1, t_2 \}} \# \bar{Q}$ (Proposition \ref{proposition_gluing_solution}). As $H_{\bar{Q}}(t_1)$ and $H_{\bar{Q}}(t_2)$ are finite, Assumption \ref{assumption_finite_entropy_IF} is satisfied for the continuous problem on $\intervalleff{t_1}{t_2}$ and therefore we can apply Proposition \ref{proposition_true_convexity_bdp} to get
\begin{equation*}
H_{\bar{Q}}((1-s)t_1 + st_2) \leqslant (1-s) H_{\bar{Q}}(t_1) + s H_{\bar{Q}}(t_2). \qedhere 
\end{equation*}     
\end{proof} 

\section{Equivalence with the parametric formulation of the Euler equation}
\label{section_equivalence_formulation}

In this section we will explain why our non-parametric formulation is equivalent to Brenier's parametric one. From the way we build it, it is clear that our formulation admits more potential solutions than Brenier's one, so the only technical point will be to show that, if the boundary data are in a parametric form, it is possible to parametrize the \emph{a priori} non-parametric solution of the continuous problem.

\bigskip

Let us take $\Afr$ a polish space and consider $\theta \in \Prob(\Afr)$ a Borel probability measure on $\Afr$. We will assume that we have two families (the initial and the final) $(\rho^\alpha_i)_{\alpha \in \Afr}$ and $(\rho^\alpha_f)_{\alpha \in \Afr}$ of probabilities measures on $\Omega$ indexed by $\Afr$. We denote by $\Pbc : \Afr \to \Gamma_{\{ 0, 1 \}} = \Prob(\Omega)^2$ the parametrization of the boundary conditions, simply defined by $\Pbc(\alpha) = (\rho^\alpha_i, \rho^\alpha_f)$ and assume that it is measurable. We assume that the boundary data satisfy the incompressibility condition, i.e. 
\begin{equation*}
\int_\Afr \rho^\alpha_i \ddr \theta(\alpha) = \Leb \ \text{ and } \ \int_\Afr \rho^\alpha_f \ddr \theta(\alpha) = \Leb. 
\end{equation*}   
Translated in our language, if we set $\gamma := \Pbc \# \theta$, we simply impose that $m_0(\gamma) = m_1(\gamma) = \Leb$.

A measurable family $(\rho^\alpha_t, \vbf^\alpha_t)_{(\alpha,t) \in \Afr \times \intervalleff{0}{1}}$ indexed by $\alpha$ and $t$ such that, for $\theta$-a.e. $\alpha$, $(t \mapsto \rho^\alpha_t) \in \Gamma$ and $\vbf^\alpha_t \in L^2(\Omega, \R^d, \rho^\alpha_t)$ for a.e. $t$, is said to be admissible if  
\begin{equation*}
\begin{cases}
\rho^\alpha_0 = \rho^\alpha_i \text{ and } \rho^\alpha_1 = \rho^\alpha_f & \text{for } \theta \text{-a.e. } \alpha, \\
\partial_t \rho^\alpha_t + \nabla \cdot (\rho^\alpha_t \vbf^\alpha_t) = 0 & \text{in a weak sense with no-flux boundary conditions for } \theta \text{-a.e. } \alpha, \\ 
\dst{\int_\Afr \rho^\alpha_t  \ddr \theta(\alpha) = \Leb} & \text{for all } t \in \intervalleff{0}{1}.
\end{cases}
\end{equation*}   
The first equation corresponds to the temporal boundary conditions, the second one is the continuity equation while the last one is the coding of the incompressibility. If $(\rho^\alpha_t, \vbf^\alpha_t)_{(\alpha,t) \in \Afr \times \intervalleff{0}{1}}$ is an admissible family, we define its (parametrized) action $\A_P$ by 
\begin{equation*}
\A_P(\rho, \vbf) := \int_\Afr \int_0^1 \int_\Omega \frac{1}{2} |\vbf^\alpha_t(x)|^2 \ddr \rho^\alpha_t(x) \ddr t \ddr \theta(\alpha)
\end{equation*} 
and its parametrized averaged entropy $H_P(\rho, \vbf) : \intervalleff{0}{1} \to \R$ by, for any $t \in \intervalleff{0}{1}$, 
\begin{equation*}
H_P(\rho, \vbf)(t) := \int_\Afr \Hen(\rho^\alpha_t) \ddr \theta(\alpha).
\end{equation*}
The first proposition is very simple: it asserts that every parametric family can be seen as an non parametric one. In the sequel, we define the boundary conditions $\gamma \in \Probin(\Gamma_{\{ 0,1 \}})$ for the non-parametric problem by $\gamma := \Pbc \# \theta$.

\begin{prop}
\label{proposition_parametric_to_nonparametric}
Let $(\rho^\alpha_t, \vbf^\alpha_t)_{(\alpha,t) \in \Afr \times \intervalleff{0}{1}}$ be an admissible family. Then there exists $Q \in \Proba(\Gamma)$ such that $\A(Q) \leqslant \A_P(\rho, \vbf)$ and $H_Q(t) = H_P(\rho, \vbf)(t)$ for any $t \in \intervalleff{0}{1}$.
\end{prop}

\begin{proof}
Let $P : \Afr \to \Gamma$, defined by $P(\alpha) = (t \mapsto \rho^\alpha_t)$ be the parametrization. We  set $Q := P \# \theta$ and leave it to the reader to check that this choice works (Theorem \ref{theorem_equivalence_A_WMD} might be useful). 
\end{proof}

The reverse proposition is slightly more difficult to prove: it asserts that one can always build a parametric family from a non-parametric $W_2$-traffic plan in such a way that the global action and the total entropy decrease. In particular, it implies together with Proposition \ref{proposition_parametric_to_nonparametric} that (provided that the boundary conditions are in a parametric form) the solution of the continuous problem \eqref{equation_continuous_problem} with minimal total entropy can be parametrized.  

\begin{prop}
Let $Q \in \Proba(\Gamma)$. Then there exists an admissible family $(\rho^\alpha_t, \vbf^\alpha_t)_{(\alpha,t) \in \Afr \times \intervalleff{0}{1}}$ such that $\A_P(\rho, \vbf) \leqslant \A(Q)$ and $H_P(\rho, \vbf)(t) \leqslant H_Q(t)$ for any $t \in \intervalleff{0}{1}$.  
\end{prop}

\begin{proof}
Let us disintegrate $Q$ w.r.t. to $e_{\{ 0,1 \}} = (e_0, e_1)$. We obtain a family $(Q_{\rho_0, \rho_1})_{\rho_0, \rho_1}$ of $W_2$-traffic plans indexed by $(\rho_0, \rho_1) \in \Gamma_{\{ 0,1 \}} = \Prob(\Omega)^2$. We define the curve $\rho^\alpha_t$ as the average of all the curves in $\Gamma$ w.r.t. to $Q_{\rho^\alpha_i, \rho^\alpha_f}$: for any $t \in \intervalleff{0}{1}$ and any $\alpha$ for which $Q_{\rho^\alpha_i, \rho^\alpha_f}$ is defined (and this property holds for $\theta$-a.e. $\alpha$), we set 
\begin{equation*}
\rho^\alpha_t := m_t \left( Q_{\rho^\alpha_i, \rho^\alpha_f} \right).
\end{equation*}   
By definition of disintegration, $e_{\{ 0,1 \}} \# Q_{\rho^\alpha_i, \rho^\alpha_f}$ is a Dirac mass at the point $(\rho^\alpha_i, \rho^\alpha_f)$, thus the boundary conditions are satisfied. The incompressibility condition is just a consequence of the incompressibility of $Q$: for any $a \in C(\Omega)$, 
\begin{align*}
\int_\Afr \left( \int_\Omega a(x) \ddr \rho^\alpha_t(x) \right) \ddr \theta(\alpha) & = \int_\Afr \left( \int_{\Gamma} \left[ \int_\Omega a(x) \ddr \rho_t(x) \right] \ddr Q_{\rho^\alpha_i, \rho^\alpha_f}(\rho) \right) \ddr \theta(\alpha) \\
& = \int_{\Gamma_{\{ 0,1 \}}} \left( \int_{\Gamma} \left[ \int_\Omega a(x) \ddr \rho_t(x) \right] \ddr Q_{\rho_0, \rho_1}(\rho) \right) \ddr \gamma(\rho_0, \rho_1) \\
& = \int_\Gamma \left( \int_\Omega a(x) \ddr \rho_t(x) \right) \ddr Q(\rho) \\
& = \int_\Omega a(x) \ddr x. 
\end{align*}
To handle the action, we use the fact that $A$ is convex and l.s.c. Thus, thanks to Jensen's inequality, for $\theta$-a.e. $\alpha$,
\begin{equation*}
A(\rho^\alpha) \leqslant \int_\Gamma A(\rho) \ddr Q_{\rho_i^\alpha, \rho_f^\alpha}(\rho).
\end{equation*}
Integrating w.r.t. $\theta$, we end up with
\begin{equation*}
\int_\Afr A(\rho^\alpha) \ddr \theta(\alpha) \leqslant \A(Q).
\end{equation*}
We consider only the case $\A(Q) < + \infty$ (else there is nothing to prove). Thus, for $\theta$-a.e. $\alpha$ the quantity $A(\rho^\alpha)$ is finite. By Theorem \ref{theorem_equivalence_A_WMD}, we can find for each $\alpha$ a family $(\vbf^\alpha_t)_{t \in \intervalleff{0}{1}}$ of functions $\Omega \to \R^d$ such that the continuity equation is satisfied, $\vbf^\alpha_t \in L^2(\Omega, \R^d, \rho^\alpha_t)$ for a.e. $t$ and such that the following identity holds 
\begin{equation*}
\int_0^1 \int_\Omega \frac{1}{2} |\vbf^\alpha_t(x)|^2 \ddr \rho^\alpha_t(x) \ddr t = \int_\Afr \int_0^1 \frac{1}{2} |\dot{\rho}^\alpha_t|^2 \ddr t. 
\end{equation*}  
Therefore, we see that the family $(\rho^\alpha_t, \vbf^\alpha_t)_{(\alpha,t) \in \Afr \times \intervalleff{0}{1}}$ is admissible and, integrating the last equality w.r.t. $\theta$, that $\A_P(\rho, \vbf) \leqslant \A(Q)$. 

To get the inequality involving the entropy, we use the fact that the functional $\Hen$ is convex and l.s.c. on $\Prob(\Omega)$, thus by Jensen's inequality,  
\begin{equation*}
\Hen \left(  m_t \left( Q_{\rho^\alpha_i, \rho^\alpha_f} \right) \right) \leqslant \int_{\Gamma} \Hen(\rho_t) \ddr Q_{\rho^\alpha_i, \rho^\alpha_f}(\rho).
\end{equation*} 
Integrating w.r.t. $\theta$ leads to the announced inequality. 
\end{proof}

\section*{Acknowledgments}

The author acknowledges the support of ANR project ISOTACE (ANR-12-MONU-0013). He also thanks Filippo Santambrogio, Aymeric Baradat, Paul Pegon and Yann Brenier for fruitful discussions and advice.


\bibliographystyle{plain}
\bibliography{bibliography}

\end{document}